\documentclass[10pt,a4paper]{article}
    \usepackage{amsmath,amssymb,amsthm,mathtools}
    \usepackage{enumerate,enumitem}
    \usepackage[all]{xy}
    \usepackage[margin=25mm]{geometry}
    \usepackage{titlesec}
        \titleformat{\section}{\normalfont\large\bf}{\thesection.}{1ex}{\centering}
        \titleformat{\subsection}[runin]{\normalfont\bf}{\thesubsection.}{1ex}{}[.]
    \theoremstyle{plain}
        \newtheorem{theorem}{Theorem}[section]
        \newtheorem{lemma}[theorem]{Lemma}

    \theoremstyle{definition}
        
    \theoremstyle{remark}
        
    \DeclareMathOperator{\sign}{sign}
    \numberwithin{equation}{section}
\begin{document}
\begin{center}\large
    {\bf Notions of solution and weak-strong uniqueness criteria \\ for the Navier-Stokes equations in Lorentz spaces} \\  \ \\
    Joseph P.\ Davies\footnote{University of Sussex, Brighton, UK, {\em jd535@sussex.ac.uk}} and Gabriel S. Koch\footnote{University of Sussex, Brighton, UK, {\em g.koch@sussex.ac.uk}}
\end{center}
\begin{abstract}
    For initial data $f\in L^{2}(\mathbb{R}^n)$ ($n\geq 2$), we prove that if $p\in(n,\infty]$, any solution \linebreak $u\in L_{t}^{\infty}L_{x}^{2}\cap L_{t}^{2}H_{x}^{1}\cap L_{t}^{\frac{2p}{p-n}}L_{x}^{p,\infty}$ to the Navier-Stokes equations satisfies the energy equality, and that such a solution $u$ is unique among all solutions $v\in L_{t}^{\infty}L_{x}^{2}\cap L_{t}^{2}H_{x}^{1}$ satisfying the energy inequality.  This extends well-known results due to G. Prodi (1959) and J. Serrin (1963), which treated the Lebesgue space $L_{x}^{p}$ rather than the larger Lorentz (and `weak Lebesgue') space $L_{x}^{p,\infty}$.  In doing so, we also prove the equivalence of various notions of solutions in $L_{x}^{p,\infty}$, generalizing in particular a result proved for the Lebesgue setting in Fabes-Jones-Riviere (1972).
\end{abstract}
\section{Introduction}
It has been known since the pioneering work of J. Leray \cite{leray1934} that certain weak solutions to the Navier-Stokes equations with initial data in the natural energy space $L^2(\mathbb{R}^n)$ always exist for all time.  These solutions (now known as `Leray-Hopf' solutions due to the later contribution of E. Hopf \cite{hopf1951}) moreover satisfy an energy inequality which implies that they belong to the space $L_{t}^{\infty}L_{x}^{2}\cap L_{t}^{2}H_{x}^{1}$.  When $n=2$, such solutions are known to be unique (see \cite{lady4}), while for $n\geq 3$, uniqueness of such solutions is not known without additional assumptions.  One has, for example, the well-known early results of G. Prodi \cite{prodi1959} and J. Serrin \cite{serrin1963} for any $n\geq 2$ that if, for some fixed initial data, there exists a Leray-Hopf solution  which belongs moreover to the space $L_{t}^{\frac{2p}{p-n}}L_{x}^{p}$ for some $p>n$, then it is the only Leray-Hopf solution for that data. (This is now known as well for the difficult endpoint $p=n$, see  \cite{ess} for $n=3$ and its generalization in \cite{dongdu} to $n>3$.)

In this paper, we extend the uniqueness results of \cite{prodi1959,serrin1963} to the so-called `weak Lebesgue' setting, namely the Lorentz spaces $L^{p,\infty}(\mathbb{R}^n)$.  In order to do so, we work with  various notions of `solution' (including in particular the Leray-Hopf type of weak solution) which we show in Theorem \ref{intro-equivalence} to be equivalent under our assumptions.  Such equivalences are well-known in the Lebesgue setting -- for example, see \linebreak \cite[Theorem 2.1]{fabes1972} in the work of Fabes-Jones-Riviere which relates the notions of `weak' and `mild' solutions in such settings,  which we  generalize in Theorem \ref{fjrgen} (along with Theorem \ref{projected-problem}) to similar weak Lebesgue settings.  The full set of equivalent notions which we address in the weak Lebesgue setting is described in Theorem \ref{intro-equivalence} below.

We point out here the recent work of T. Barker \cite{barker2018} which establishes a similar type of uniqueness result (see \cite[Proposition 1.6]{barker2018}) for $n=3$ if there exists a solution in the mixed space-time Lorentz space $L_{t}^{\frac{2p}{p-3},s}L_{x}^{p,s}$ for some $p>3$ and $s<\infty$, and for $s=\infty$ but only if its norm in that space is sufficiently small.  In contrast, by considering the Lebesgue setting in time only, we are able to remove the requirement of smallness under the assumption of existence in $L_{t}^{\frac{2p}{p-n}}L_{x}^{p,\infty}$ ($p>n$), for any $n\geq 2$.  Moreover, the result in \cite{barker2018} is established as a by-product of the more general result \cite[Theorem 1.3]{barker2018} in the setting of certain larger Besov spaces, which is proved using more sophisticated tools than those in \cite{prodi1959,serrin1963}. Our proof is more direct, and largely follows the method in \cite{serrin1963}.

After this work was completed, it was pointed out to us by P. G. Lemari\'e-Rieusset that, when $n=3$, the statement of our weak-strong uniqueness theorem (Theorem \ref{intro-weak-strong} below) was indicated by a comment in the work \cite{lemarie3}. Specifically, \cite[Proposition 5]{lemarie3} 
is a weak-strong uniqueness result which is proved in a more general type of space denoted by $\dot X_r$, and there is a remark just before that proposition that ``$L^{3/r,\infty} \subset \dot X_r$ for $r<3/2$'' (however the embedding is not proved there).  In our proof, we do not require the use of spaces of the form $\dot X_r$, and give a full proof directly in the space $L^{p,\infty}(\mathbb{R}^n)$ (and not only for $n=3$).  Our method, however, is essentially the same as that in \cite{lemarie3}.

\subsection{General discussion}
On the space-time domain $(0,T)\times\mathbb{R}^{n}$, we consider the \textbf{\textit{linearised Navier-Stokes equations}}
\begin{equation}\label{initial-value-problem}
    \left\{\begin{array}{l}u_{j}'-\Delta u_{j}+\nabla_{k}F_{jk}+\nabla_{j}P=0 \quad (1\leq j \leq n), \\ \nabla\cdot u=0, \\ u(0)=f, \end{array}\right.
\end{equation}
where the \textbf{\textit{Navier-Stokes equations}} are realised by taking $F_{jk}=u_{j}u_{k}$. At the formal level, we can eliminate the pressure term $P$ by applying the Leray projection
\begin{equation}
    \mathbb{P}_{ij} = \delta_{ij}-\frac{\nabla_{i}\nabla_{j}}{\Delta}
\end{equation}
onto solenoidal vector fields, which transforms \eqref{initial-value-problem} into the heat equation
\begin{equation}\label{intro-heat-equation}
    \left\{\begin{array}{l}u_{i}'-\Delta u_{i}+\nabla_{k}\mathbb{P}_{ij}F_{jk}=0 \quad (1\leq i \leq n), \\ \nabla\cdot u=0, \\ u(0)=f, \end{array}\right.
\end{equation}
for which we have the corresponding integral equation
\begin{equation}\label{intro-integral-equation}
    u_{i}(t,x) = \int_{\mathbb{R}^{n}}\Phi(t,x-y)f_{i}(y)\,\mathrm{d}y - \int_{0}^{t}\int_{\mathbb{R}^{n}}\nabla_{k}\Phi(t-s,x-y)[\mathbb{P}_{ij}F_{jk}(s)](y)\,\mathrm{d}y\,\mathrm{d}s
\end{equation}
for $1\leq i \leq n$, where the heat kernel $\Phi$ is defined by
\begin{equation}
    \Phi(t,x) := \frac{1}{{(4\pi t)}^{n/2}}e^{-{|x|}^{2}/4t} = \frac{1}{{(2\pi)}^{n}}\int_{\mathbb{R}^{n}}e^{\mathrm{i}x\cdot\xi-t{|\xi|}^{2}}\,\mathrm{d}\xi \quad \text{for }(t,x)\in(0,\infty)\times\mathbb{R}^{n}.
\end{equation}

We will work here in the context\footnote{In \cite{davies2021}, we develop local well-posedness (with $F=u\otimes u$) of \eqref{intro-integral-equation} in this context, but essentially with the projection $\mathbb{P}$ on the kernel $\Phi$ rather than on $F$; this small discrepancy can be resolved using estimates in \cite{davies2021}.} of \textbf{\textit{Lorentz spaces}} $L^{p,q}(\mathbb{R}^{n})$, where for a measurable function $f$ on $\mathbb{R}^{n}$ we define
\begin{equation}
    \lambda_{f}(y) := |\{x\in\mathbb{R}^{n}\text{ : }|f(x)|>y\}| \quad \text{for }y\in(0,\infty),
\end{equation}
\begin{equation}
    f^{*}(\tau) := \sup\{y\in(0,\infty)\text{ : }\lambda_{f}(y)>\tau\} \quad \text{for }\tau\in(0,\infty),
\end{equation}
\begin{equation}
    {\|f\|}_{L^{p,q}(\mathbb{R}^{n})}^{*} := \left\{\begin{array}{ll}{\left(\frac{q}{p}\int_{0}^{\infty}{\left[\tau^{1/p}f^{*}(\tau)\right]}^{q}\,\frac{\mathrm{d}\tau}{\tau}\right)}^{1/q} & p\in[1,\infty),\,q\in[1,\infty), \\ \sup_{\tau\in(0,\infty)}\tau^{1/p}f^{*}(\tau) & p\in[1,\infty],\,q=\infty, \end{array}\right.
\end{equation}
and the quasinorms ${\|\cdot\|}_{L^{p,q}(\mathbb{R}^{n})}^{*}$ satisfy
\begin{equation}
    {\|f\|}_{L^{p,p}(\mathbb{R}^{n})}^{*} = {\|f\|}_{L^{p}(\mathbb{R}^{n})} \quad \text{for }p\in[1,\infty],
\end{equation}
\begin{equation}
    {\|f\|}_{L^{p,q_{2}}(\mathbb{R}^{n})}^{*} \leq {\|f\|}_{L^{p,q_{1}}(\mathbb{R}^{n})}^{*} \quad \text{for }p\in[1,\infty) \text{ and }1\leq q_{1}\leq q_{2}\leq\infty.
\end{equation}
In order to understand how \eqref{intro-integral-equation} relates to \eqref{initial-value-problem} and \eqref{intro-heat-equation}, we require an understanding of the Leray projection $\mathbb{P}$. We will achieve this by defining $\mathbb{P}_{ij}:=\delta_{ij}+\mathcal{R}_{i}\mathcal{R}_{j}$, where the \textbf{\textit{Riesz transform}} $\mathcal{R}$ is the unique linear map
\begin{equation}
    \mathcal{R}:\cup_{p\in(1,\infty)}\left(L^{1}(\mathbb{R}^{n})+L^{p,\infty}(\mathbb{R}^{n})\right)\rightarrow\cup_{p\in(1,\infty)}\left(L^{1}(\mathbb{R}^{n})+L^{p,\infty}(\mathbb{R}^{n})\right)
\end{equation}
which satisfies
\begin{equation}
    \mathcal{R}f = \mathcal{F}^{-1}\left[\xi\mapsto\frac{-\mathrm{i}\xi}{|\xi|}\mathcal{F}f(\xi)\right] \quad \text{for all }f\in L^{2}(\mathbb{R}^{n}),
\end{equation}
\begin{equation}
    \left\{\begin{array}{ll}{\|\mathcal{R}f\|}_{L^{p,q}(\mathbb{R}^{n})}^{*} \lesssim_{n,p,q} {\|f\|}_{L^{p,q}(\mathbb{R}^{n})}^{*} \text{ and }\langle\mathcal{R}f,g\rangle=-\langle f,\mathcal{R}g\rangle \\ \text{for all }p\in(1,\infty),\text{ }q\in[1,\infty],\text{ }f\in L^{p,q}(\mathbb{R}^{n})\text{ }g\in L^{p',q'}(\mathbb{R}^{n}), \end{array}\right.
\end{equation}
where $\mathcal{F}$ is the Fourier transform, and we use the notation $\langle f,g\rangle:=\int_{\mathbb{R}^{n}}f(x)g(x)\,\mathrm{d}x$.

When attempting to relate the integral equation \eqref{intro-integral-equation} to the differential equations \eqref{initial-value-problem} and \eqref{intro-heat-equation}, the following subtlety arises regarding measurability: if $p\in(1,\infty)$, $(E,\mathcal{E})$ is a measurable space, and $u$ is a measurable function on $E\times\mathbb{R}^{n}$ satisfying $u(t)\in L^{p,\infty}(\mathbb{R}^{n})$ for all $t\in E$, then it is not immediately obvious that the Riesz transform defines a measurable function $\mathcal{R}u$ on $E\times\mathbb{R}^{n}$. We therefore define a weaker notion of measurability, by saying that a function $u:E\rightarrow L_{\mathrm{loc}}^{1}(\mathbb{R}^{n})$ is \textbf{\textit{weakly* measurable}} if the map $t\mapsto\langle u(t),\phi\rangle$ is measurable for each $\phi\in C_{c}^{\infty}(\mathbb{R}^{n})$. For $p,q\in[1,\infty)$ satisfying $(p=1\Rightarrow q=1)$, we can prove the following:
\begin{itemize}
    \item If $u:E\rightarrow L^{p',q'}(\mathbb{R}^{n})$ is weakly* measurable, then $t\mapsto\langle u(t),\phi\rangle$ is measurable for each $\phi\in L^{p,q}(\mathbb{R}^{n})$, and the map $t\mapsto{\|u(t)\|}_{{\left(L^{p,q}(\mathbb{R}^{n})\right)}^{*}}:=\sup_{{\|\phi\|}_{L^{p,q}(\mathbb{R}^{n})}^{*}\leq1}\langle u(t),\phi\rangle$ is measurable, where the norm ${\|\cdot\|}_{{\left(L^{p,q}(\mathbb{R}^{n})\right)}^{*}}$ is equivalent to the quasinorm ${\|\cdot\|}_{L^{p',q'}(\mathbb{R}^{n})}^{*}$ on $L^{p',q'}(\mathbb{R}^{n})$.
    \item If $u:E\rightarrow L^{p,q}(\mathbb{R}^{n})$ is weakly* measurable, then $t\mapsto\langle u(t),\phi\rangle$ is measurable for each $\phi\in L^{p',q'}(\mathbb{R}^{n})$, and the map $t\mapsto{\|u(t)\|}_{{\left(L^{p,q}(\mathbb{R}^{n})\right)}^{**}}:=\sup_{{\|\phi\|}_{{\left(L^{p,q}(\mathbb{R}^{n})\right)}^{*}}\leq1}\langle u(t),\phi\rangle$ is measurable, where the norm ${\|\cdot\|}_{{\left(L^{p,q}(\mathbb{R}^{n})\right)}^{**}}$ is equivalent to the quasinorm ${\|\cdot\|}_{L^{p,q}(\mathbb{R}^{n})}^{*}$ on $L^{p,q}(\mathbb{R}^{n})$.
    \item \textbf{\textit{Pettis' theorem:}} If $u:E\rightarrow L^{p,q}(\mathbb{R}^{n})$ is weakly* measurable, then there exist measurable functions $u_{m}:E\rightarrow L^{p,q}(\mathbb{R}^{n})$ with finite image, satisfying ${\|u_{m}(t)\|}_{{\left(L^{p,q}(\mathbb{R}^{n})\right)}^{**}}\leq{\|u(t)\|}_{{\left(L^{p,q}(\mathbb{R}^{n})\right)}^{**}}$ and ${\|u_{m}(t)-u(t)\|}_{{\left(L^{p,q}(\mathbb{R}^{n})\right)}^{**}}\overset{m\rightarrow\infty}{\rightarrow}0$ pointwise.
\end{itemize}
For $T\in(0,\infty]$, and $\alpha,p,q\in[1,\infty]$ satisfying $(p=1\Rightarrow q=1)$ and $(p=\infty\Rightarrow q=\infty)$, we write $\mathcal{L}_{p,q}^{\alpha;*}(T)$ to denote (equivalence classes of) weakly* measurable functions $u:(0,T)\rightarrow L^{p,q}(\mathbb{R}^{n})$ satisfying ``${\|u\|}_{L^{p,q}(\mathbb{R}^{n})}^{*}\in L^{\alpha}((0,T))$'', understood in the sense  that the quasinorm ${\|u(t)\|}_{L^{p,q}(\mathbb{R}^{n})}^{*}$ is equivalent to a norm which is a measurable function of $t$. We write $\mathcal{L}_{p,q}^{\alpha;*}(T_{-}):=\cap_{T'\in(0,T)}\mathcal{L}_{p,q}^{\alpha;*}(T)$, we remove the $q$ from the notation in the case $p=q$, and we remove the * from the notation if $u$ defines a measurable function on $(0,T)\times\mathbb{R}^{n}$.

We can now state the following theorem, which relates various formulations of the linearised Navier-Stokes equations.
\begin{theorem}\label{intro-equivalence}
    Assume that $T\in(0,\infty]$, $p_{0},p,\tilde{p}\in(1,\infty)$, $f\in L^{p_{0},\infty}(\mathbb{R}^{n})$ satisfies $\langle f_{i},\nabla_{i}\phi\rangle=0$ for all $\phi\in C_{c}^{\infty}(\mathbb{R}^{n})$, $F\in\mathcal{L}_{\tilde{p},\infty}^{1;*}(T_{-})$, and $u\in\mathcal{L}_{p,\infty}^{1;*}(T_{-})$. Then the following are equivalent.
    \begin{itemize}
        \item \textbf{\textit{Weak formulation:}} $\int_{0}^{T}\langle u_{j}(t),\nabla_{j}\psi(t)\rangle\,\mathrm{d}t$ for all $\psi\in C_{c}^{\infty}((0,T)\times\mathbb{R}^{n})$, and there exists some $P\in\mathcal{L}_{\tilde{p},\infty}^{1;*}(T)$ such that
        \begin{equation}
            \langle f_{j},\phi_{j}(0)\rangle + \int_{0}^{T}\left(\langle u_{j}(t),\phi_{j}'(t)+\Delta\phi_{j}(t)\rangle + \langle F_{jk}(t),\nabla_{k}\phi_{j}(t)\rangle + \langle P(t),\nabla_{j}\phi_{j}(t)\rangle\right)\,\mathrm{d}t = 0
        \end{equation}
        for all $\phi\in C_{c}^{\infty}([0,T)\times\mathbb{R}^{n})$. Moreover, $P(t)=\mathcal{R}_{i}\mathcal{R}_{j}F(t)$ for almost every $t\in(0,T)$.
        \item \textbf{\textit{Projected formulation:}} $\int_{0}^{T}\langle u_{j}(t),\nabla_{j}\psi(t)\rangle\,\mathrm{d}t$ for all $\psi\in C_{c}^{\infty}((0,T)\times\mathbb{R}^{n})$, and
        \begin{equation}\label{projform}
            \langle f_{i},\phi_{i}(0)\rangle + \int_{0}^{T}\left(\langle u_{i}(t),\phi_{i}'(t)+\Delta\phi_{i}(t)\rangle + \langle F_{jk}(t),\mathbb{P}_{ij}\nabla_{k}\phi_{i}(t)\rangle\right)\,\mathrm{d}t = 0
        \end{equation}
        for all $\phi\in C_{c}^{\infty}([0,T)\times\mathbb{R}^{n})$.
        \item \textbf{\textit{Mild formulation:}} $[u(t)](x)=v(t,x)$ for almost every $(t,x)\in(0,T)\times\mathbb{R}^{n}$, where $v$ is the (almost everywhere defined) measurable function on $(0,T)\times\mathbb{R}^{n}$ given by
        \begin{equation}
            v_{i}(t,x) := \int_{\mathbb{R}^{n}}\Phi(t,x-y)f_{i}(y)\,\mathrm{d}y - \int_{0}^{t}\int_{\mathbb{R}^{n}}\nabla_{k}\Phi(t-s,x-y)[\mathbb{P}_{ij}F_{jk}(s)](y)\,\mathrm{d}y\,\mathrm{d}s
        \end{equation}
        for almost every $(t,x)\in(0,T)\times\mathbb{R}^{n}$.
        \item \textbf{\textit{Very weak formulation:}} $\int_{0}^{T}\langle u_{j}(t),\nabla_{j}\psi(t)\rangle\,\mathrm{d}t$ for all $\psi\in C_{c}^{\infty}((0,T)\times\mathbb{R}^{n})$, and
        \begin{equation}
            \langle f_{i},\theta(0)\phi_{i}\rangle + \int_{0}^{T}\left(\langle u_{i}(t),\theta'(t)\phi_{i}+\theta(t)\Delta\phi_{i}\rangle + \langle F_{jk}(t),\theta(t)\nabla_{k}\phi_{j}\rangle\right)\,\mathrm{d}t = 0
        \end{equation}
        for all $\theta\in C_{c}^{\infty}([0,T))$ and $\phi\in C_{c,\sigma}^{\infty}(\mathbb{R}^{n})$, where the subscript $\sigma$ means that $\nabla\cdot\phi=0$.
    \end{itemize}
\end{theorem}
Measurability issues in the above formulations can be addressed by approximating the test functions using Pettis' theorem. The weak/projected/mild formulations correspond to equations \eqref{initial-value-problem}, \eqref{intro-heat-equation} and \eqref{intro-integral-equation} respectively.

Of particular interest are solutions to the Navier-Stokes equations which belong to the energy class $L_{t}^{\infty}L_{x}^{2}\cap L_{t}^{2}H_{x}^{1}$. More precisely, for $n\geq 2$ and $T\in(0,\infty]$, we define $\mathcal{H}_{T}$ to be the space of (equivalence classes of) weakly* measurable functions $u:(0,T)\rightarrow H^{1}(\mathbb{R}^{n})\subseteq L_{\mathrm{loc}}^{1}(\mathbb{R}^{n})$ satisfying $u\in\mathcal{L}_{2}^{\infty;*}(T_{-})$ and $\nabla u\in\mathcal{L}_{2}^{2;*}(T_{-})$, where we observe that weak* measurability of $u$ implies weak* measurability of $\nabla u$. By virtue of the identification $\mathcal{L}_{2}^{2;*}(T')\cong{\left(L^{2}((0,T')\times\mathbb{R}^{n})\right)}^{*}\cong L^{2}((0,T')\times\mathbb{R}^{n})$, we may identify $u$ and $\nabla u$ with measurable functions $u\in\mathcal{L}_{2}^{\infty}(T_{-})$ and $\nabla u\in\mathcal{L}_{2}^{2}(T_{-})$. For each $p\in(n,\infty]$ we have the \textbf{\textit{Sobolev inequality}} ${\|u\|}_{L^{\frac{2p}{p-2},2}(\mathbb{R}^{n})}\lesssim_{n,p}{\|u\|}_{L^{2}(\mathbb{R}^{n})}^{1-\frac{n}{p}}{\|\nabla u\|}_{L^{2}(\mathbb{R}^{n})}^{\frac{n}{p}}$, from which we deduce that $\mathcal{H}_{T}\subseteq\cap_{p\in(n,\infty]}\mathcal{L}_{\frac{2p}{p-2},2}^{2p/n}(T_{-})$. In particular, if $u\in\mathcal{H}_{T}$ and $F_{jk}=u_{j}u_{k}$ then the equivalence result of Theorem \ref{intro-equivalence} holds. Independently of Theorem \ref{intro-equivalence}, we will prove the following.

\pagebreak
\begin{theorem}\label{intro-weak-strong}
    Let $T\in(0,\infty]$, and assume that $f\in L^{2}(\mathbb{R}^{n})$ satisfies $\langle f_{i},\nabla_{i}\phi\rangle=0$ for all $\phi\in C_{c}^{\infty}(\mathbb{R}^{n})$.
    \begin{enumerate}[label=(\roman*)]
        \item If $u\in\mathcal{H}_{T}$ satisfies the mild formulation of the Navier-Stokes equations with initial data $f$, then $u$ satisfies the continuity condition
        \begin{equation}\label{intro-continuity-condition}
            \left\{\begin{array}{l}\text{there exists a subset }\Omega\subseteq(0,T)\text{ of total measure such that} \\ \langle u(t),\phi\rangle\rightarrow\langle f,\phi\rangle\text{ for all }\phi\in L^{2}(\mathbb{R}^{n})\text{ as }t\rightarrow0\text{ along }\Omega. \end{array}\right.
        \end{equation}
        \item Assume that $p\in(n,\infty]$, and that $u\in\mathcal{H}_{T}\cap\mathcal{L}_{p,\infty}^{2p/(p-n)}(T_{-})$ and $v\in\mathcal{H}_{T}$ satisfy the projected formulation of the Navier-Stokes equations with intial data $f$, with $u$ and $v$ both satisfying the continuity condition \eqref{intro-continuity-condition}. Then for almost every $t\in(0,T)$ we have
        \begin{equation}
        \begin{aligned}
            \langle u_{i}(t),v_{i}(t)\rangle &= {\|f\|}_{L^{2}(\mathbb{R}^{n})}^{2} - 2\int_{0}^{t}\langle\nabla_{j}u_{i}(s),\nabla_{j}v_{i}(s)\rangle\,\mathrm{d}s \\
            &\qquad -\int_{0}^{t}\langle{[(u\cdot\nabla)u]}_{i}(s),v_{i}(s)\rangle\,\mathrm{d}s -\int_{0}^{t}\langle u_{i}(s),{[(v\cdot\nabla)v]}_{i}(s)\rangle\,\mathrm{d}s.
        \end{aligned}
        \end{equation}
        In particular, $u$ satisfies the \textbf{\textit{energy equality}}
        \begin{equation}
            {\|f\|}_{L^{2}(\mathbb{R}^{n})}^{2} = {\|u(t)\|}_{L^{2}(\mathbb{R}^{n})}^{2} + 2\int_{0}^{t}{\|\nabla u(s)\|}_{L^{2}(\mathbb{R}^{n})}^{2}\,\mathrm{d}s \quad \text{for a.e.\ }t\in(0,T).
        \end{equation}
        If we make the additional assumption that $v$ satisfies the \textbf{\textit{energy inequality}}
        \begin{equation}
            {\|f\|}_{L^{2}(\mathbb{R}^{n})}^{2}\geq{\|v(t)\|}_{L^{2}(\mathbb{R}^{n})}^{2} + 2\int_{0}^{t}{\|\nabla v(s)\|}_{L^{2}(\mathbb{R}^{n})}^{2}\,\mathrm{d}s \quad \text{for a.e.\ }t\in(0,T),
        \end{equation}
        then $u(t,x)=v(t,x)$ for almost every $(t,x)\in(0,T)\times\mathbb{R}^{n}$.
    \end{enumerate}
\end{theorem}
Part (ii) of the previous theorem generalises the work of Prodi \cite{prodi1959} and Serrin \cite{serrin1963}, in which the assumption $u\in\mathcal{L}_{p,\infty}^{2p/(p-n)}(T_{-})$ is replaced by the stronger assumption $u\in\mathcal{L}_{p}^{2p/(p-n)}(T_{-})$. In the particular case $n=2$, we have $\mathcal{H}_{T}\subseteq\cap_{p\in(2,\infty)}\mathcal{L}_{p,\infty}^{2p/(p-n)}(T_{-})$ by the Sobolev inequality. For initial data $f\in J(\mathbb{R}^{n}):=\overline{C_{c}^{\infty}(\mathbb{R}^{n})}^{{\|\cdot\|}_{L^{2}(\mathbb{R}^{n})}}$, Hopf \cite{hopf1951} constructs a solution $u\in\mathcal{H}_{\infty}$ to the very weak formulation of the Navier-Stokes equations which satisfies the energy inequality.
\subsection{Lorentz spaces}
For a measurable function $f$ on $\mathbb{R}^{n}$ we define the auxilliary functions
\begin{equation}
    \lambda_{f}(y) := |\{x\in\mathbb{R}^{n}\text{ : }|f(x)|>y\}| \quad \text{for }y\in(0,\infty),
\end{equation}
\begin{equation}
    f^{*}(\tau) := \sup\{y\in(0,\infty)\text{ : }\lambda_{f}(y)>\tau\} \quad \text{for }\tau\in(0,\infty),
\end{equation}
\begin{equation}
    f^{**}(\tau) := \sup_{|A|\geq\tau}\frac{1}{|A|}\int_{A}|f(x)|\,\mathrm{d}x \overset{*}{=} \frac{1}{\tau}\int_{0}^{\tau}f^{*}(\eta)\,\mathrm{d}\eta \quad \text{for }\tau\in(0,\infty).
\end{equation}
The \textbf{\textit{Lorentz quasinorms}} ${\|f\|}_{L^{p,q}(\mathbb{R}^{n})}^{*}$ and the \textbf{\textit{Lorentz norms}} ${\|f\|}_{L^{p,q}(\mathbb{R}^{n})}$ are then defined by
\begin{equation}
    {\|f\|}_{L^{p,q}(\mathbb{R}^{n})}^{*} := \left\{\begin{array}{ll}{\left(\frac{q}{p}\int_{0}^{\infty}{\left[\tau^{1/p}f^{*}(\tau)\right]}^{q}\,\frac{\mathrm{d}\tau}{\tau}\right)}^{1/q} & p\in[1,\infty),\,q\in[1,\infty), \\ \sup_{\tau\in(0,\infty)}\tau^{1/p}f^{*}(\tau) & p\in[1,\infty],\,q=\infty, \end{array}\right.
\end{equation}
\begin{equation}
    {\|f\|}_{L^{p,q}(\mathbb{R}^{n})} := \left\{\begin{array}{ll}{\left(\frac{q}{p}\int_{0}^{\infty}{\left[\tau^{1/p}f^{**}(\tau)\right]}^{q}\,\frac{\mathrm{d}\tau}{\tau}\right)}^{1/q} & p\in(1,\infty),\,q\in(1,\infty), \\ \sup_{\tau\in(0,\infty)}\tau^{1/p}f^{**}(\tau) & p\in(1,\infty],\,q=\infty. \end{array}\right.
\end{equation}
We have the following properties, many of which generalise familiar properties of Lebesgue spaces.  As we discuss in \cite{davies2021}, the majority of these can be found in (or derived easily from statements in) \cite{hunt1966}, while part of property (viii) follows by arguing for example along the lines of (\cite[Theorem 1.26]{mouhot2017}).

\pagebreak
\begin{enumerate}[label=(\roman*)]
    \item ${\|f\|}_{L^{p,p}(\mathbb{R}^{n})}^{*} = {\|f\|}_{L^{p}(\mathbb{R}^{n})}$ for $p\in[1,\infty]$.
    \item ${\|f\|}_{L^{p,q_{2}}(\mathbb{R}^{n})}^{*} \leq {\|f\|}_{L^{p,q_{1}}(\mathbb{R}^{n})}^{*}$ for $p\in[1,\infty)$ and $1\leq q_{1}\leq q_{2}\leq\infty$.
    \item $f^{*}(\tau)\leq f^{**}(\tau)$, and ${\|f\|}_{L^{p,q}(\mathbb{R}^{n})}^{*}\leq {\|f\|}_{L^{p,q}(\mathbb{R}^{n})} \leq p'{\|f\|}_{L^{p,q}(\mathbb{R}^{n})}^{*}$ for $p\in(1,\infty]$ and $q\in[1,\infty]$ with $(p=\infty\Rightarrow q=\infty)$.
    \item ${\left\|\int_{\mathbb{R}}f(t,\cdot)\,\mathrm{d}t\right\|}_{L^{p,q}(\mathbb{R}^{n})} \leq \int_{\mathbb{R}}{\|f(t,\cdot)\|}_{L^{p,q}(\mathbb{R}^{n})}\,\mathrm{d}t$ for a measurable function $f$ on $\mathbb{R}\times\mathbb{R}^{n}$, with $p\in(1,\infty]$ and $q\in[1,\infty]$ satisfying $(p=\infty\Rightarrow q=\infty)$.
    \item $\int_{\mathbb{R}^{n}}|f(x)g(x)|\,\mathrm{d}x\leq\int_{0}^{\infty}f^{*}(\tau)g^{*}(\tau)\,\mathrm{d}\tau$, ${(fg)}^{*}(\tau)\leq f^{**}(\tau)g^{**}(\tau)$, and ${(fg)}^{**}(\tau)\leq{\|f\|}_{L^{\infty}(\mathbb{R}^{n})}g^{**}(\tau)$. (In applications, the first two of these can be combined with H\"{o}lder's inequality).
    \item ${(f*g)}^{**}(\tau)\leq{\|f\|}_{L^{1}(\mathbb{R}^{n})}g^{**}(\tau)$, and ${\|f*g\|}_{L^{r,s}(\mathbb{R}^{n})}\lesssim_{p,q,s}{\|f\|}_{L^{q,s}(\mathbb{R}^{n})}{\|g\|}_{L^{p,\infty}(\mathbb{R}^{n})}$ for $p,q,r\in(1,\infty)$, $s\in[1,\infty]$ and $1+\frac{1}{r}=\frac{1}{p}+\frac{1}{q}$.
    \item ${\|f\|}_{L^{p,\infty}(\mathbb{R}^{n})}\leq{\|f\|}_{L^{p_{0},\infty}(\mathbb{R}^{n})}^{1-\theta}{\|f\|}_{L^{p_{1},\infty}(\mathbb{R}^{n})}^{\theta}$ and ${\|f\|}_{L^{p,1}(\mathbb{R}^{n})}\lesssim_{p,p_{0},p_{1}}{\|f\|}_{L^{p_{0},\infty}(\mathbb{R}^{n})}^{1-\theta}{\|f\|}_{L^{p_{1},\infty}(\mathbb{R}^{n})}^{\theta}$ when $1<p_{0}<p_{1}\leq\infty$, $\theta\in(0,1)$ and $\frac{1}{p}=\frac{1-\theta}{p_{0}}+\frac{\theta}{p_{1}}$.
    \item If $p,q\in[1,\infty)$ satisfy $(p=1\Rightarrow q=1)$, then $L^{p,q}(\mathbb{R}^{n})$ is separable, contains $C_{c}^{\infty}(\mathbb{R}^{n})$ as a dense subset, and satisfies continuity of translation ${\|\phi(\cdot-h)-\phi(\cdot)\|}_{L^{p,q}(\mathbb{R}^{n})}^{*}\overset{h\rightarrow0}{\rightarrow}0$ and approximation of identity ${\left\|\int_{\mathbb{R}^{n}}\epsilon^{-n}f(y/\epsilon)(\phi(\cdot-y)-\phi(\cdot))\,\mathrm{d}y\right\|}_{L^{p,q}(\mathbb{R}^{n})}^{*}\overset{\epsilon\rightarrow0}{\rightarrow}0$ for $f\in L^{1}(\mathbb{R}^{n})$ and $\phi\in L^{p,q}(\mathbb{R}^{n})$. We note that continuity of translation and approximation of identity also hold if the space $L^{p,q}(\mathbb{R}^{n})$ is replaced by the space $C_{b,u}^{0}(\mathbb{R}^{n})$ of bounded, uniformly continuous functions, equipped with the supremum norm.
    \item If $p_{0}\neq p_{1}$, $r_{0}\neq r_{1}$, and $T:L^{p_{0},q_{0}}(\mathbb{R}^{n})+L^{p_{1},q_{1}}(\mathbb{R}^{n})\rightarrow L^{r_{0},s_{0}}(\mathbb{R}^{n})+L^{r_{1},s_{1}}(\mathbb{R}^{n})$ is a function satisfying
    \begin{equation}
        |T(f+g)| \leq K(|Tf|+|Tg|) \text{ a.e., } \quad {\|Tf\|}_{L^{r_{i},s_{i}}(\mathbb{R}^{n})}^{*} \leq B_{i}{\|f\|}_{L^{p_{i},q_{i}}(\mathbb{R}^{n})}^{*} \text{ for }i\in\{0,1\},
    \end{equation}
    then for $\theta\in(0,1)$, $\frac{1}{p_{\theta}}=\frac{1-\theta}{p_{0}}+\frac{\theta}{p_{1}}$ and $\frac{1}{r_{\theta}}=\frac{1-\theta}{r_{0}}+\frac{\theta}{r_{1}}$ we have $L^{p_{\theta},\infty}(\mathbb{R}^{n})\subseteq L^{p_{0},q_{0}}(\mathbb{R}^{n})+L^{p_{1},q_{1}}(\mathbb{R}^{n})$ and
    \begin{equation}
        {\|Tf\|}_{L^{r_{\theta},s}(\mathbb{R}^{n})}^{*} \leq B_{\theta}{\|f\|}_{L^{p_{\theta},q}(\mathbb{R}^{n})}^{*} \quad \text{for }q\leq s,
    \end{equation}
    where $B_{\theta}$ depends on $\theta$, the Lorentz indices and the constants $K,B_{0},B_{1}$.
\end{enumerate}
We also recall \textbf{\textit{Hardy's inequalities}} (\cite{hunt1966}, p.\ 256), which state that for $p\in(0,\infty)$ and $q\in[1,\infty)$ we have
\begin{equation}
\begin{aligned}
    {\left(\int_{0}^{\infty}{\left[\int_{0}^{t}t^{-\frac{1}{p}}\phi(s)\,\frac{\mathrm{d}s}{s}\right]}^{q}\,\frac{\mathrm{d}t}{t}\right)}^{\frac{1}{q}} &\leq p{\left(\int_{0}^{\infty}{\left[s^{-\frac{1}{p}}\phi(s)\right]}^{q}\,\frac{\mathrm{d}s}{s}\right)}^{\frac{1}{q}}, \\
    {\left(\int_{0}^{\infty}{\left[\int_{t}^{\infty}t^{\frac{1}{p}}\phi(s)\,\frac{\mathrm{d}s}{s}\right]}^{q}\,\frac{\mathrm{d}t}{t}\right)}^{\frac{1}{q}} &\leq p{\left(\int_{0}^{\infty}{\left[s^{\frac{1}{p}}\phi(s)\right]}^{q}\,\frac{\mathrm{d}s}{s}\right)}^{\frac{1}{q}}.
\end{aligned}
\end{equation}

For $p,q\in[1,\infty]$ satisfying $(p=1\Rightarrow q=1)$ and $(p=\infty\Rightarrow q=\infty)$, the quasinormed space $\left(L^{p,q}(\mathbb{R}^{n}),{\|\cdot\|}_{L^{p,q}(\mathbb{R}^{n})}^{*}\right)$ is normable (either by a Lebesgue norm or a Lorentz norm). For a normable quasinormed space $(X,{\|\cdot\|}_{X}^{*})$, the \textbf{\textit{dual space}} $X^{*}$ is defined to be the space of linear maps $F:X\rightarrow\mathbb{R}$ for which the \textbf{\textit{dual norm}} ${\|F\|}_{X^{*}}:=\sup_{{\|\phi\|}_{X}^{*}\leq1}|F(\phi)|$ is finite. If $Y$ is a linear subspace of $X$, and $F:Y\rightarrow\mathbb{R}$ is a linear map satisfying $|F(\phi)|\leq M{\|\phi\|}_{X}^{*}$ for $\phi\in Y$, then by the Hahn-Banach theorem there exists $\tilde{F}\in X^{*}$ with ${\|\tilde{F}\|}_{X^{*}}\lesssim_{X}M$ and $\tilde{F}|_{Y}=F$. It follows that a normable quasinorm ${\|\cdot\|}_{X}^{*}$ is equivalent to its \textbf{\textit{bidual norm}} ${\|\phi\|}_{X^{**}}:=\sup_{{\|F\|}_{X^{*}}\leq1}|F(\phi)|$ for $\phi\in X$.

For the sake of completeness, we prove the following standard result concerning the dual space of $L^{p,q}(\mathbb{R}^{n})$ for $p<\infty$.
\begin{lemma}
    For $p,q\in[1,\infty)$ satisfying $(p=1\Rightarrow q=1)$, the map $u\mapsto F_{u}$ given by $F_{u}(\phi)=\langle u,\phi\rangle$ defines a Banach space isomorphism $L^{p',q'}(\mathbb{R}^{n})\rightarrow{\left(L^{p,q}(\mathbb{R}^{n})\right)}^{*}$.
\end{lemma}
\begin{proof}
    By property (v) of Lorentz spaces, the map $u\mapsto F_{u}$ is well-defined, linear, injective and bounded. It remains to show that this map is surjective and that its inverse is bounded, so for each $F\in{\left(L^{p,q}(\mathbb{R}^{n})\right)}^{*}$ we seek $u\in L^{p',q'}(\mathbb{R}^{n})$ satisfying ${\|u\|}_{L^{p',q'}(\mathbb{R}^{n})}^{*} \lesssim_{p,q} {\|F\|}_{{\left(L^{p,q}(\mathbb{R}^{n})\right)}^{*}}$ and
    \begin{equation}\label{duality-goal}
        F(\phi) = \int_{\mathbb{R}^{n}}u(x)\phi(x)\,\mathrm{d}x\text{ for all }\phi\in L^{p,q}(\mathbb{R}^{n}).
    \end{equation}
    The case $q=1$ is proved in (\cite{hunt1966}, p.\ 261), so we assume that $p,q\in(1,\infty)$. If $p,q\in(1,\infty)$ and $F\in{\left(L^{p,q}(\mathbb{R}^{n})\right)}^{*}$, then $|F(\phi)|\leq {\|F\|}_{{\left(L^{p,q}(\mathbb{R}^{n})\right)}^{*}}{\|\phi\|}_{L^{p,q}(\mathbb{R}^{n})}^{*}\leq {\|F\|}_{{\left(L^{p,q}(\mathbb{R}^{n})\right)}^{*}}{\|\phi\|}_{L^{p,1}(\mathbb{R}^{n})}^{*}$ for all $\phi\in L^{p,1}(\mathbb{R}^{n})$, so by the case $q=1$ there exists $u\in L^{p,\infty}(\mathbb{R}^{n})$ satisfying
    \begin{equation}
        F(\phi) = \int_{\mathbb{R}^{n}}u(x)\phi(x)\,\mathrm{d}x\text{ for all }\phi\in L^{p,1}(\mathbb{R}^{n}).
    \end{equation}
    We make the claim (stated without proof in \cite{hunt1966}) that ${\|u\|}_{L^{p',q'}(\mathbb{R}^{n})}^{*}\lesssim_{p,q}{\|F\|}_{{\left(L^{p,q}(\mathbb{R}^{n})\right)}^{*}}$, which implies (\ref{duality-goal}). We argue along the lines of (\cite{bennett1988}, Theorem IV.4.7). Consider the simple functions $u_{m}=\sign(u)\min\{m,2^{-m}\lfloor2^{m}|u|\rfloor\}$, so that ${\|u_{m}\|}_{L^{p',q'}(\mathbb{R}^{n})}^{*}\overset{m\rightarrow\infty}{\rightarrow}{\|u\|}_{L^{p',q'}(\mathbb{R}^{n})}^{*}$ by monotone convergence. Since $u_{m}^{*}$ is a decreasing function, we have
    \begin{equation}
    \begin{aligned}
        {\left({\|u_{m}\|}_{L^{p',q'}(\mathbb{R}^{n})}^{*}\right)}^{q'} &\lesssim_{p,q} \int_{0}^{\infty}{\left[t^{\frac{1}{p'}}u_{m}^{*}(t)\right]}^{q'}\,\frac{\mathrm{d}t}{t} \\
        &= \int_{0}^{\infty}t^{\frac{q'}{p'}-1}u_{m}^{*}(t)^{q'-1}u_{m}^{*}(t)\,\mathrm{d}t \\
        &\leq \int_{0}^{\infty}\left(\int_{t/e}^{t}t^{\frac{q'}{p'}-1}u_{m}^{*}(s)^{q'-1}\,\frac{\mathrm{d}s}{s}\right)u_{m}^{*}(t)\,\mathrm{d}t \\
        &\lesssim_{p,q} \int_{0}^{\infty}\left(\int_{t/e}^{t}s^{\frac{q'}{p'}-1}u_{m}^{*}(s)^{q'-1}\,\frac{\mathrm{d}s}{s}\right)u_{m}^{*}(t)\,\mathrm{d}t \\
        &\leq \int_{0}^{\infty}\psi_{m}(t)u_{m}^{*}(t)\,\mathrm{d}t,
    \end{aligned}
    \end{equation}
    where $\psi_{m}(t) := \int_{t/e}^{\infty}s^{\frac{q'}{p'}-1}u_{m}^{*}(s)^{q'-1}\,\frac{\mathrm{d}s}{s}$ defines a decreasing continuous function on $(0,\infty)$. Since $u_{m}$ is a simple function, the map
    \begin{equation}
        \tau_{m}(x) := |\{z\in\mathbb{R}^{n}\text{ : }|u_{m}(z)|>|u_{m}(x)|\}|+|\{z\in\mathbb{R}^{n}\text{ : }|z|<|x|,|u_{m}(z)|=|u_{m}(x)|\}|
    \end{equation}
    defines a measure preserving transformation $\tau_{m}:(\mathbb{R}^{n},\mathrm{d}x)\rightarrow((0,\infty),\mathrm{d}t)$ satisfying $|u_{m}|=u_{m}^{*}\circ\tau_{m}$ almost everywhere. Writing $\phi_{m}=\sign(u)\psi_{m}\circ\tau_{m}$, we consider the simple functions $\phi_{m,j}=\sign(u)\min\{j,2^{-j}\lfloor2^{j}|\phi_{m}|\rfloor\}$, so by monotone convergence we have
    \begin{equation}
    \begin{aligned}
        \int_{0}^{\infty}\psi_{m}(t)u_{m}^{*}(t)\,\mathrm{d}t &= \int_{\mathbb{R}^{n}}\phi_{m}(x)u_{m}(x)\,\mathrm{d}x \\
        &= \lim_{j\rightarrow\infty}\int_{\mathbb{R}^{n}}\phi_{m,j}(x)u_{m}(x)\,\mathrm{d}x \\
        &\leq \lim_{j\rightarrow\infty}\int_{\mathbb{R}^{n}}\phi_{m,j}(x)u(x)\,\mathrm{d}x \\
        &= \lim_{j\rightarrow\infty}F(\phi_{m,j}) \\
        &\leq \lim_{j\rightarrow\infty}{\|F\|}_{{\left(L^{p,q}(\mathbb{R}^{n})\right)}^{*}}{\|\phi_{m,j}\|}_{L^{p,q}(\mathbb{R}^{n})}^{*} \\
        &= {\|F\|}_{{\left(L^{p,q}(\mathbb{R}^{n})\right)}^{*}}{\|\phi_{m}\|}_{L^{p,q}(\mathbb{R}^{n})}^{*}.
    \end{aligned}
    \end{equation}
    We use $\phi_{m}^{*}=\psi_{m}$, the substitution $r=es$ and Hardy's inequality to estimate
    \begin{equation}
    \begin{aligned}
        {\|\phi_{m}\|}_{L^{p,q}(\mathbb{R}^{n})}^{*} &\lesssim_{p,q} {\left(\int_{0}^{\infty}{\left[t^{\frac{1}{p}}\psi_{m}(t)\right]}^{q}\,\frac{\mathrm{d}t}{t}\right)}^{\frac{1}{q}} \\
        &= {\left(\int_{0}^{\infty}{\left[\int_{t/e}^{\infty}t^{\frac{1}{p}}s^{\frac{q'}{p'}-1}u_{m}^{*}(s)^{q'-1}\,\frac{\mathrm{d}s}{s}\right]}^{q}\,\frac{\mathrm{d}t}{t}\right)}^{\frac{1}{q}} \\
        &\lesssim_{p,q} {\left(\int_{0}^{\infty}{\left[\int_{t}^{\infty}t^{\frac{1}{p}}r^{\frac{q'}{p'}-1}u_{m}^{*}(r/e)^{q'-1}\,\frac{\mathrm{d}r}{r}\right]}^{q}\,\frac{\mathrm{d}t}{t}\right)}^{\frac{1}{q}} \\
        &\lesssim_{p,q} {\left(\int_{0}^{\infty}{\left[r^{\frac{q'-1}{p'}}u_{m}^{*}(r/e)^{q'-1}\right]}^{q}\,\frac{\mathrm{d}r}{r}\right)}^{\frac{1}{q}} \\
        &\lesssim_{p,q} {\left(\int_{0}^{\infty}{\left[s^{\frac{q'-1}{p'}}u_{m}^{*}(s)^{q'-1}\right]}^{q}\,\frac{\mathrm{d}s}{s}\right)}^{\frac{1}{q}} \\
        &= {\left(\int_{0}^{\infty}{\left[s^{\frac{1}{p}}u_{m}^{*}(s)\right]}^{q'}\,\frac{\mathrm{d}s}{s}\right)}^{\frac{1}{q}} \\
        &\lesssim_{p,q} {\left({\|u_{m}\|}_{L^{p',q'}(\mathbb{R}^{n})}^{*}\right)}^{\frac{q'}{q}}.
    \end{aligned}
    \end{equation}
    We conclude that
    \begin{equation}
    \begin{aligned}
        {\left({\|u_{m}\|}_{L^{p',q'}(\mathbb{R}^{n})}^{*}\right)}^{q} \lesssim_{p,q} \int_{0}^{\infty}\psi_{m}(t)u_{m}^{*}(t)\,\mathrm{d}t &\leq {\|F\|}_{{\left(L^{p,q}(\mathbb{R}^{n})\right)}^{*}}{\|\phi_{m}\|}_{L^{p,q}(\mathbb{R}^{n})}^{*} \\
        &\lesssim_{p,q}{\|F\|}_{{\left(L^{p,q}(\mathbb{R}^{n})\right)}^{*}}{\left({\|u_{m}\|}_{L^{p',q'}(\mathbb{R}^{n})}^{*}\right)}^{\frac{q'}{q}},
    \end{aligned}
    \end{equation}
    so ${\|u_{m}\|}_{L^{p',q'}(\mathbb{R}^{n})}^{*}\lesssim_{p,q} {\|F\|}_{{\left(L^{p,q}(\mathbb{R}^{n})\right)}^{*}}$ and hence ${\|u\|}_{L^{p',q'}(\mathbb{R}^{n})}^{*}\lesssim_{p,q} {\|F\|}_{{\left(L^{p,q}(\mathbb{R}^{n})\right)}^{*}}$.
\end{proof}
\section{Formulations of the linearised Navier-Stokes equations}
On the space-time domain $(0,T)\times\mathbb{R}^{n}$, we will discuss various formulations of the \textbf{\textit{linearised Navier-Stokes equations}}
\begin{equation}
    \left\{\begin{array}{l}u_{j}'-\Delta u_{j}+\nabla_{k}F_{jk}+\nabla_{j}P=0, \\ \nabla\cdot u=0, \\ u(0)=f, \end{array}\right.
\end{equation}
where the \textbf{\textit{Navier-Stokes equations}} are realised by taking $F_{jk}=u_{j}u_{k}$. Each formulation will be preceded by a discussion of some of the technical preliminaries relevant to that formulation.
\subsection{Weak* measurability and the weak formulation}
For a measurable space $(E,\mathcal{E})$, a function $u:E\rightarrow L_{\mathrm{loc}}^{1}(\mathbb{R}^{n})$ is said to be \textbf{\textit{weakly* measurable}} if the map $t\mapsto\langle u(t),\phi\rangle$ is measurable for all $\phi\in C_{c}^{\infty}(\mathbb{R}^{n})$. The following results make particular use of property (viii) from our discussion of Lorentz spaces.
\begin{lemma}
    For $p,q\in[1,\infty]$ with $(p=1\Rightarrow q=1)$ and $(p=\infty\Rightarrow q=\infty)$, a function $u:E\rightarrow L^{p,q}(\mathbb{R}^{n})$ is weakly* measurable if and only if the map $t\mapsto\langle u(t),\phi\rangle$ is measurable for all $\phi\in L^{p',q'}(\mathbb{R}^{n})$.
\end{lemma}
\begin{proof}
    Define the cutoff function $\rho_{R}(x)=\rho(x/R)$ for $R\in\mathbb{Q}_{>0}$, where $\rho\in C_{c}^{\infty}(\mathbb{R}^{n})$ satisfies $\rho(x)=1$ for $|x|<1$ and $\rho(x)=0$ for $|x|>2$. Define also the approximate identity $\eta_{\epsilon}(x)=\epsilon^{-n}\eta(x/\epsilon)$ for $\epsilon\in\mathbb{Q}_{>0}$, where $\eta\in C_{c}^{\infty}(\mathbb{R}^{n})$ satisfies $\int_{\mathbb{R}^{n}}\eta(x)\,\mathrm{d}x=1$. Then the function
    \begin{equation}
        \langle u(t),\phi\rangle = \lim_{\epsilon\rightarrow0}\langle u(t),\eta_{\epsilon}*\phi\rangle = \lim_{\epsilon\rightarrow0}\lim_{R\rightarrow\infty}\langle u(t),\rho_{R}\cdot(\eta_{\epsilon}*\phi)\rangle
    \end{equation}
    is measurable, where the limit $R\rightarrow\infty$ follows from dominated convergence, and the limit $\epsilon\rightarrow0$ follows from approximation of identity (if necessary, use Fubini's theorem to throw the convolution onto whichever one of $u(t),\phi$ belongs to a separable Lorentz space).
\end{proof}

\pagebreak
\begin{lemma}\label{measurable-norms}
    Let $p,q\in[1,\infty)$ satisfy $(p=1\Rightarrow q=1)$.
    \begin{enumerate}[label=(\roman*)]
        \item If $u:E\rightarrow L^{p',q'}(\mathbb{R}^{n})$ is weakly* measurable, then $t\mapsto{\|u(t)\|}_{{\left(L^{p,q}(\mathbb{R}^{n})\right)}^{*}}$ is measurable.
        \item If $u:E\rightarrow L^{p,q}(\mathbb{R}^{n})$ is weakly* measurable, then $t\mapsto{\|u(t)\|}_{{\left(L^{p,q}(\mathbb{R}^{n})\right)}^{**}}$ is measurable.
    \end{enumerate}
\end{lemma}
\begin{proof}
    \begin{enumerate}[label=(\roman*)]
        \item
        Let $A$ be a countable dense subset of $\{\phi\in L^{p,q}(\mathbb{R}^{n})\text{ : }{\|\phi\|}_{L^{p,q}(\mathbb{R}^{n})}^{*}\leq1\}$, so the function ${\|u(t)\|}_{{\left(L^{p,q}(\mathbb{R}^{n})\right)}^{*}}=\sup_{\phi\in A}\langle u(t),\phi\rangle$ is measurable.
        \item
        If there exists a countable family $A\subseteq B:=\{F\in L^{p',q'}(\mathbb{R}^{n})\text{ : }{\|F\|}_{{\left(L^{p,q}(\mathbb{R}^{n})\right)}^{*}}\leq1\}$, such that for each $F\in B$ there exists a sequence ${(F_{m})}_{m\geq1}$ in $A$ satisfying $\langle F_{m},\phi\rangle\overset{m\rightarrow\infty}{\rightarrow}\langle F,\phi\rangle$ for all $\phi\in L^{p,q}(\mathbb{R}^{n})$, then the function ${\|u(t)\|}_{{\left(L^{p,q}(\mathbb{R}^{n})\right)}^{**}}=\sup_{F\in A}\langle F,u(t)\rangle$ is measurable. We construct $A$ as follows. Let ${(\phi_{m})}_{m\geq1}$ be dense in $L^{p,q}(\mathbb{R}^{n})$, and observe that each family $\left\{{(\langle F,\phi_{l}\rangle)}_{l\leq m}\text{ : }F\in B\right\}\subseteq\mathbb{R}^{m}$ has a dense subset $\left\{{(\langle F_{k,m},\phi_{l}\rangle)}_{l\leq m}\text{ : }k\geq1\right\}$. We claim that $A={(F_{k,m})}_{k,m\geq1}$ is as required. Given $F\in B$, for each $m\geq1$ there exists $k(m)$ such that $\left|\langle F_{k(m),m}-F,\phi_{l}\rangle\right|<\frac{1}{m}$ for all $l\leq m$, so for any $\phi\in L^{p,q}(\mathbb{R}^{n})$ and $\epsilon>0$ we can choose $l(\epsilon)$ with ${\|\phi_{l(\epsilon)}-\phi\|}_{L^{p,q}(\mathbb{R}^{n})}^{*}<\epsilon$ to obtain
        \begin{equation}
            \limsup_{m\rightarrow\infty}\left|\langle F_{k(m),m}-F,\phi\rangle\right|\leq\limsup_{m\rightarrow\infty}\left(\left|\langle F_{k(m),m}-F,\phi_{l(\epsilon)}\rangle\right|+\left|\langle F_{k(m),m}-F,\phi_{l(\epsilon)}-\phi\rangle\right|\right) \leq 2\epsilon.
        \end{equation}
    \end{enumerate}
\end{proof}
\begin{lemma}
    (Pettis' theorem). Let $p,q\in[1,\infty)$ satisfy $(p=1\Rightarrow q=1)$. If the function $u:E\rightarrow L^{p,q}(\mathbb{R}^{n})$ is weakly* measurable, then there exist measurable functions $u_{m}:E\rightarrow L^{p,q}(\mathbb{R}^{n})$ with finite image, satisfying ${\|u_{m}(t)\|}_{{\left(L^{p,q}(\mathbb{R}^{n})\right)}^{**}}\leq{\|u(t)\|}_{{\left(L^{p,q}(\mathbb{R}^{n})\right)}^{**}}$ and ${\|u_{m}(t)-u(t)\|}_{{\left(L^{p,q}(\mathbb{R}^{n})\right)}^{**}}\overset{m\rightarrow\infty}{\rightarrow}0$ pointwise.
\end{lemma}
\begin{proof}
    Let $A$ be a countable dense subset of $L^{p,q}(\mathbb{R}^{n})$, let $\mathbb{Q}A=\left\{qa\text{ : }q\in\mathbb{Q},\,a\in A\right\}$, and let ${(f_{j})}_{j\geq0}$ be an enumeration of $\mathbb{Q}A$ with $f_{0}=0$. For each $m\geq0$ and $t\in E$, define the non-empty set
    \begin{equation}
        K(m,t) := \left\{j\in\{0,\cdots,m\}\text{ : }{\|f_{j}\|}_{{\left(L^{p,q}(\mathbb{R}^{n})\right)}^{**}}\leq{\|u(t)\|}_{{\left(L^{p,q}(\mathbb{R}^{n})\right)}^{**}}\right\}
    \end{equation}
    and the integer
    \begin{equation}
        k(m,t) := \min\left\{j\in K(m,t)\text{ : }{\|f_{j}-u(t)\|}_{{\left(L^{p,q}(\mathbb{R}^{n})\right)}^{**}}=\min_{i\in K(m,t)}{\|f_{i}-u(t)\|}_{{\left(L^{p,q}(\mathbb{R}^{n})\right)}^{**}}\right\}.
    \end{equation}
    Then $u_{m}(t):=f_{k(m,t)}$ defines a measurable function $u_{m}:E\rightarrow L^{p,q}(\mathbb{R}^{n})$ with finite image, satisfying ${\|u_{m}(t)\|}_{{\left(L^{p,q}(\mathbb{R}^{n})\right)}^{**}}\leq{\|u(t)\|}_{{\left(L^{p,q}(\mathbb{R}^{n})\right)}^{**}}$ and ${\|u_{m}(t)-u(t)\|}_{{\left(L^{p,q}(\mathbb{R}^{n})\right)}^{**}}\overset{m\rightarrow\infty}{\rightarrow}0$ pointwise.
\end{proof}
For $T\in(0,\infty]$, and $\alpha,p,q\in[1,\infty]$ satisfying $(p=1\Rightarrow q=1)$ and $(p=\infty\Rightarrow q=\infty)$, we write $\mathcal{L}_{p,q}^{\alpha;*}(T)$ to denote (equivalence classes of) weakly* measurable functions $u:(0,T)\rightarrow L^{p,q}(\mathbb{R}^{n})$ satisfying ``${\|u\|}_{L^{p,q}(\mathbb{R}^{n})}^{*}\in L^{\alpha}((0,T))$'', understood in the sense  that the quasinorm ${\|u(t)\|}_{L^{p,q}(\mathbb{R}^{n})}^{*}$ is equivalent to a norm which is a measurable function of $t$. If $\alpha,p,q\in[1,\infty)$ satisfy $(p=1\Rightarrow q=1)$, and $u\in\mathcal{L}_{p,q}^{\alpha;*}(T)$, then the approximating functions from Pettis' theorem are simple (supported on sets of fnite measure), and satisfy ${\left\|t\mapsto{\|u_{m}(t)-u(t)\|}_{{\left(L^{p,q}(\mathbb{R}^{n})\right)}^{**}}\right\|}_{L^{\alpha}((0,T))}\overset{m\rightarrow\infty}{\rightarrow}0$ by dominated convergence. We write $\mathcal{L}_{p,q}^{\alpha;*}(T_{-}):=\cap_{T'\in(0,T)}\mathcal{L}_{p,q}^{\alpha;*}(T)$, we remove the $q$ from the notation in the case $p=q$, and we remove the * from the notation if $u$ defines a measurable function on $(0,T)\times\mathbb{R}^{n}$.

Sometimes we will want to replace the interval $(0,T)$ by the interval $(-\infty,T)$; in such situations, we will write $\tilde{\mathcal{L}}$ instead of $\mathcal{L}$. One such situation is the following approximation lemma, which will play an important role in establishing energy estimates for the Navier-Stokes equations.

\pagebreak
\begin{lemma}\label{approximation-lemma}
    Let $\alpha\in(1,\infty)$, let $p,q\in[1,\infty)$ satisfy $(p=1\Rightarrow q=1)$, and let $\eta_{\epsilon}(t)=\epsilon^{-1}\eta(t/\epsilon)$ for some $\eta\in L^{1}((-1,1))$ satisfying $\int_{-1}^{1}\eta(t)\,\mathrm{d}t=1$. For $T\in(-\infty,\infty]$, $u\in\tilde{\mathcal{L}}_{p,q}^{\alpha}(T_{-})$ and $U\in\tilde{\mathcal{L}}_{p',q'}^{\alpha'}(T_{-})$, define the mollified functions
    \begin{equation}
        u^{\epsilon}(t,x) := \int_{-\infty}^{T}\eta_{\epsilon}(t-s)u(s,x)\,\mathrm{d}s, \quad U^{\epsilon}(t,x) := \int_{-\infty}^{T}\eta_{\epsilon}(t-s)U(s,x)\,\mathrm{d}s
    \end{equation}
    for almost every $(t,x)\in(-\infty,T-\epsilon)\times\mathbb{R}^{n}$. Then for all $T'\in(-\infty,T)$ we have
    \begin{enumerate}[label=(\roman*)]
        \item $u^{\epsilon}\rightarrow u$ in $\tilde{\mathcal{L}}_{p,q}^{\alpha}(T')$ as $\epsilon\rightarrow0$,
        \item $\langle U^{\epsilon},u^{\epsilon}\rangle\rightarrow\langle U,u\rangle$ in $L^{1}((-\infty,T'))$ as $\epsilon\rightarrow0$.
    \end{enumerate}
\end{lemma}
\begin{proof}
    Approximating $u$ via Pettis' theorem, it suffices to consider the case where the image of $u:(-\infty,T)\rightarrow L^{p,q}(\mathbb{R}^{n})$ is contained within a finite dimensional subspace $X\subseteq L^{p,q}(\mathbb{R}^{n})$. Let ${(f_{i})}_{i=1}^{m}$ be a basis for $X$, and let ${(\theta_{i})}_{i=1}^{m}$ be the basis for $X^{*}$ satisfying $\theta_{i}(f_{j})=\delta_{ij}$. By the Hahn-Banach theorem, there exist $F_{i}\in L^{p',q'}(\mathbb{R}^{n})$ satisfying $\langle F_{i},f\rangle=\theta_{i}(f)$ for all $f\in X$. Then $u=\sum_{i=1}^{m}u_{i}f_{i}$, and $\langle U,f\rangle=\langle\sum_{i=1}^{m}U_{i}F_{i},f\rangle$ for all $f\in X$, where the functions $u_{i}:=\langle F_{i},u\rangle$ and $U_{i}:=\langle U,f_{i}\rangle$ are measurable on $(-\infty,T)$. It therefore suffices to consider the functions $\mathbf{u}={(u_{i})}_{i=1}^{m}\in\cap_{T'\in(-\infty,T)}L^{\alpha}((-\infty,T');\mathbb{R}^{m})$ and $\mathbf{U}={(U_{i})}_{i=1}^{m}\in\cap_{T'\in(-\infty,T)}L^{\alpha'}((-\infty,T');\mathbb{R}^{m})$. For $T'\in(-\infty,T-\epsilon)$, we use H\"{o}lder's inequality and continuity of translation to estimate
    \begin{equation}
    \begin{aligned}
        {\|\mathbf{u}^{\epsilon}-\mathbf{u}\|}_{L^{\alpha}((-\infty,T');\mathbb{R}^{m})} &= {\left(\int_{-\infty}^{T'}{\left|\int_{-\epsilon}^{\epsilon}\eta_{\epsilon}(s)\left(\mathbf{u}(t-s)-\mathbf{u}(t)\right)\,\mathrm{d}s\right|}^{\alpha}\,\mathrm{d}t\right)}^{\frac{1}{\alpha}} \\
        &\leq {\left(\int_{-\infty}^{T'}{\left(\int_{-\epsilon}^{\epsilon}|\eta_{\epsilon}(s)|\,\mathrm{d}s\right)}^{\alpha-1}\left(\int_{-\epsilon}^{\epsilon}|\eta_{\epsilon}(s)|{|\mathbf{u}(t-s)-\mathbf{u}(t)|}^{\alpha}\,\mathrm{d}s\right)\,\mathrm{d}t\right)}^{\frac{1}{\alpha}} \\
        &= {\|\eta\|}_{L^{1}(\mathbb{R})}^{1-\frac{1}{\alpha}}{\left(\int_{-\epsilon}^{\epsilon}\left(\int_{-\infty}^{T'}{|\mathbf{u}(t-s)-\mathbf{u}(t)|}^{\alpha}\,\mathrm{d}t\right)|\eta_{\epsilon}(s)|\,\mathrm{d}s\right)}^{\frac{1}{\alpha}} \overset{\epsilon\rightarrow0}{\rightarrow}0,
    \end{aligned}
    \end{equation}
    and analogously ${\|\mathbf{U}^{\epsilon}-\mathbf{U}\|}_{L^{\alpha'}((-\infty,T');\mathbb{R}^{m})}\overset{\epsilon\rightarrow0}{\rightarrow}0$, so $\mathbf{U}^{\epsilon}\cdot\mathbf{u}^{\epsilon}\overset{\epsilon\rightarrow0}{\rightarrow}\mathbf{U}\cdot\mathbf{u}$ in $L^{1}((-\infty,T');\mathbb{R})$.
\end{proof}
We are now in a position to define the \textbf{\textit{weak formulation}} of the linearised Navier-Stokes equations. For $T\in(0,\infty]$, $p_{0},p,\tilde{p}\in(1,\infty)$, $f\in L^{p_{0},\infty}(\mathbb{R}^{n}))$ satisfying $\langle f_{i},\nabla_{i}\phi\rangle$ for all $\phi\in C_{c}^{\infty}(\mathbb{R}^{n})$, and $F\in\mathcal{L}_{\tilde{p},\infty}^{1;*}(T_{-})$, we say that $u\in\mathcal{L}_{p,\infty}^{1;*}(T_{-})$ is a \textbf{\textit{weak solution to the linearised Navier-Stokes equations}} if the exists $P\in\mathcal{L}_{\tilde{p},\infty}^{1;*}(T_{-})$ such that
\begin{equation}
    \left\{\begin{array}{ll}
        \langle f_{j},\phi_{j}(0)\rangle + \int_{0}^{T}\left(\langle u_{j},\phi_{j}'+\Delta\phi_{j}\rangle + \langle F_{jk},\nabla_{k}\phi_{j}\rangle + \langle P,\nabla_{j}\phi_{j}\rangle\right)\,\mathrm{d}t = 0 & \forall\,\phi\in C_{c}^{\infty}\left([0,T)\times\mathbb{R}^{n}\right) \\
        \int_{0}^{T}\langle u_{j},\nabla_{j}\psi\rangle\,\mathrm{d}t = 0 & \forall\,\psi\in C_{c}^{\infty}\left((0,T)\times\mathbb{R}^{n}\right) \\
    \end{array}\right.
\end{equation}
where measurability issues are addressed by approximating the test functions using Pettis' theorem. This definition depends a priori on the index $\tilde{p}$, since we seek a pressure term $P\in\mathcal{L}_{\tilde{p},\infty}^{1;*}(T_{-})$. However, we will eventually prove that the weak formulation is equivalent to another formulation which does not depend on $\tilde{p}$.

For $T\in(0,\infty]$, $p_{0},p\in(1,\infty)$, and $f\in L^{p_{0},\infty}(\mathbb{R}^{n}))$ satisfying $\langle f_{i},\nabla_{i}\phi\rangle$ for all $\phi\in C_{c}^{\infty}(\mathbb{R}^{n})$, we say that $u\in\mathcal{L}_{2p,\infty}^{2;*}(T_{-})$ is a \textbf{\textit{weak solution to the Navier-Stokes equations}} if $u$ is a weak  solution to the linearised Navier-Stokes equations with $F_{jk}=u_{j}u_{k}$. To show that $u_{j}u_{k}$ is weakly* measurable, we write $\langle u_{j}(t)u_{k}(t),\phi\rangle=\langle u_{j}(t),u_{k}(t)\phi\rangle$ for $\phi\in C_{c}^{\infty}(\mathbb{R}^{n})$, and we approximate the weakly* measurable function $u_{k}\phi:(0,T)\rightarrow L^{\frac{2p}{2p-1},1}(\mathbb{R}^{n})$ using Pettis' theorem.

\pagebreak
\subsection{The Riesz transform and the projected formulation}
The Riesz transform $\mathcal{R}$ gives a precise meaning to the pseudo-differential operator $\frac{-\nabla}{\sqrt{-\Delta}}$.
\begin{lemma}\label{riesz-lemma}
    There exists a unique linear map
    \begin{equation}
        \mathcal{R}:\cup_{p\in(1,\infty)}\left(L^{1}(\mathbb{R}^{n})+L^{p,\infty}(\mathbb{R}^{n})\right)\rightarrow\cup_{p\in(1,\infty)}\left(L^{1}(\mathbb{R}^{n})+L^{p,\infty}(\mathbb{R}^{n})\right)
    \end{equation}
    which satisfies
    \begin{equation}
        \mathcal{R}f = \mathcal{F}^{-1}\left[\xi\mapsto\frac{-\mathrm{i}\xi}{|\xi|}\mathcal{F}f(\xi)\right] \quad \text{for all }f\in L^{2}(\mathbb{R}^{n}),
    \end{equation}
    \begin{equation}
        \left\{\begin{array}{ll}{\|\mathcal{R}f\|}_{L^{p,q}(\mathbb{R}^{n})}^{*} \lesssim_{n,p,q} {\|f\|}_{L^{p,q}(\mathbb{R}^{n})}^{*} \text{ and }\langle\mathcal{R}f,g\rangle=-\langle f,\mathcal{R}g\rangle \\ \text{for all }p\in(1,\infty),\text{ }q\in[1,\infty],\text{ }f\in L^{p,q}(\mathbb{R}^{n})\text{ }g\in L^{p',q'}(\mathbb{R}^{n}). \end{array}\right.
    \end{equation}
\end{lemma}
Assuming Lemma \ref{riesz-lemma}, we can define the Leray projection $\mathbb{P}_{ij}:=\delta_{ij}+\mathcal{R}_{i}\mathcal{R}_{j}$ onto divergence-free vector fields, which allows us to reduce the weak formulation of the linearised Navier-Stokes equations to a projected formulation, which is independent of the pressure $P$.
\begin{theorem}\label{projected-problem}
    Let $T\in(0,\infty]$, $p_{0},p,\tilde{p}\in(1,\infty)$, $f\in L^{p_{0},\infty}(\mathbb{R}^{n})$ with $\langle f_{i},\nabla_{i}\phi\rangle$ for all $\phi\in C_{c}^{\infty}(\mathbb{R}^{n})$, and $F\in\mathcal{L}_{\tilde{p},\infty}^{1;*}(T_{-})$. Then $u\in\mathcal{L}_{p,\infty}^{1;*}(T_{-})$ is a weak solution to the linearised Navier-Stokes equations if and only if $u$ satisfies the \textbf{\textit{projected formulation}}
    \begin{equation}
        \left\{\begin{array}{ll}
            \langle f_{i},\phi_{i}(0)\rangle + \int_{0}^{T}\left(\langle u_{i},\phi_{i}'+\Delta\phi_{i}\rangle + \langle F_{jk},\mathbb{P}_{ij}\nabla_{k}\phi_{i}\rangle\right)\,\mathrm{d}t = 0 & \forall\,\phi\in C_{c}^{\infty}\left([0,T)\times\mathbb{R}^{n}\right) \\
            \int_{0}^{T}\langle u_{i},\nabla_{i}\psi\rangle\,\mathrm{d}t = 0 & \forall\,\psi\in C_{c}^{\infty}\left((0,T)\times\mathbb{R}^{n}\right), \\
        \end{array}\right.
    \end{equation}
    where $\mathbb{P}_{ij}F_{jk}\in\mathcal{L}_{\tilde{p},\infty}^{1;*}(T_{-})$ by properties of the Riesz transform. Moreover, the function $P$ in the weak formulation is uniquely determined by $P=\mathcal{R}_{j}\mathcal{R}_{k}F_{jk}$.
\end{theorem}
\begin{proof}
    Given $F\in\mathcal{L}_{\tilde{p},\infty}^{1;*}(T_{-})$, the expression $P:=\mathcal{R}_{j}\mathcal{R}_{k}F_{jk}$ defines a weakly* measurable function $P\in\mathcal{L}_{\tilde{p},\infty}^{1;*}(T_{-})$ satisfying
    \begin{equation}
        \langle P(t),\nabla_{i}\psi_{i}\rangle = \langle F_{jk}(t),\mathcal{R}_{j}\mathcal{R}_{k}\nabla_{i}\psi_{i}\rangle = \langle F_{jk},\mathcal{R}_{i}\mathcal{R}_{j}\nabla_{k}\psi_{i}\rangle \quad \text{for all }\psi\in C_{c}^{\infty}(\mathbb{R}^{n})
    \end{equation}
    for almost every $t\in(0,T)$, so if $u$ solves the projected problem then $u$ is a weak solution.

    Conversely, by considering $\phi_{j}(t,x)=\theta(t)\nabla_{j}\psi(x)$, where $\psi\in C_{c}^{\infty}(\mathbb{R}^{n})$ is fixed and $\theta\in C_{c}^{\infty}((0,T))$ is allowed to vary, if $u$ is a weak solution with pressure $P$ then
    \begin{equation}
        0=\langle F_{jk}(t),\nabla_{j}\nabla_{k}\psi\rangle + \langle P(t),\Delta\psi\rangle=-\langle F_{jk}(t),\mathcal{R}_{j}\mathcal{R}_{k}\Delta\psi\rangle+\langle P(t),\Delta\psi\rangle \quad \text{for a.e.\ }t\in(0,T)
    \end{equation}
    for all $\psi\in C_{c}^{\infty}(\mathbb{R}^{n})$, so $\langle P(t),\chi\rangle=\langle\mathcal{R}_{j}\mathcal{R}_{k}F_{jk}(t),\chi\rangle$ for almost every $t\in(0,T)$ for all $\chi\in\Delta C_{c}^{\infty}(\mathbb{R}^{n})$. By considering a countable dense subset of $\left(\Delta C_{c}^{\infty}(\mathbb{R}^{n}),{\|\cdot\|}_{L^{\tilde{p}',1}(\mathbb{R}^{n})}\right)$, it follows that $\langle P(t),\chi\rangle=\langle\mathcal{R}_{j}\mathcal{R}_{k}F_{jk}(t),\chi\rangle$ for all $\chi\in\Delta C_{c}^{\infty}(\mathbb{R}^{n})$ for almost every $t\in(0,T)$. To conclude that $P(t)=\mathcal{R}_{j}\mathcal{R}_{k}F_{jk}(t)$ for almost every $t\in(0,T)$, we make use of the following lemma:

    {\em If $P\in L_{\mathrm{loc}}^{1}(\mathbb{R}^{n})$ satisfies $\langle P,\Delta\psi\rangle$ for all $\psi\in C_{c}^{\infty}(\mathbb{R}^{n})$, then $P$ is almost everywhere equal to some $\tilde{P}\in C^{2}(\mathbb{R}^{n})$ satisfying $\Delta\tilde{P}=0$, so if $P\in L^{\tilde{p},\infty}(\mathbb{R}^{n})$ for $\tilde{p}\in(1,\infty)$ then $P=0$ almost everywhere.}

    {\em Proof of lemma.} Let $\eta_{\epsilon}(x)=\epsilon^{-n}\eta(x/\epsilon)$ for some $\eta\in C_{c}^{\infty}(\mathbb{R}^{n})$ with $\int_{\mathbb{R}^{n}}\eta(x)\,\mathrm{d}x=1$. Then $P_{\epsilon}:=P*\eta_{\epsilon}$ is a smooth function satisfying $\langle\Delta P_{\epsilon},\psi\rangle$ for all $\psi\in C_{c}^{\infty}(\mathbb{R}^{n})$, so $\Delta P_{\epsilon}=0$. By the mean value property (\cite{gilbarg1983}, Theorem 2.7), we therefore have $P_{\epsilon}(x)=\frac{1}{|B(x,R)|}\int_{B(x,R)}P_{\epsilon}(y)\,\mathrm{d}y$ for all $x\in\mathbb{R}^{n}$ and $R\in(0,\infty)$. By continuity of translation in $L^{1}(\mathbb{R}^{n})$ we have
    \begin{equation}
    \begin{aligned}
        \left|\int_{B(x,R)}(P_{\epsilon}(y)-P(y))\,\mathrm{d}y\right| &= \left|\int_{B(x,R)}\left(\int_{\mathbb{R}^{n}}(P(y-z)-P(y))\eta_{\epsilon}(z)\,\mathrm{d}z\right)\,\mathrm{d}y\right| \\
        &\leq \int_{\mathbb{R}^{n}}\left(\int_{B(x,R)}\left|P(y-z)-P(y)\right|\,\mathrm{d}y\right)|\eta_{\epsilon}(z)|\,\mathrm{d}z
        &\overset{\epsilon\rightarrow0}{\rightarrow}0
    \end{aligned}
    \end{equation}
    locally uniformly in $x\in\mathbb{R}^{n}$. Therefore, as $\epsilon\rightarrow 0$, $P_{\epsilon}$ converges to $P$ in $L_{\text{loc}}^{1}(\mathbb{R}^{n})$, and $P_{\epsilon}$ converges locally uniformly to a continuous function $\tilde{P}$ satisfying $\tilde{P}(x)=\frac{1}{|B(x,R)|}\int_{B(x,R)}P(y)\,\mathrm{d}y$. Then $P=\tilde{P}$ almost everywhere, so $\tilde{P}(x)=\frac{1}{|B(x,R)|}\int_{B(x,R)}\tilde{P}(y)\,\mathrm{d}y$, and by the mean value property we deduce that $\tilde{P}\in C^{2}(\mathbb{R}^{n})$ with $\Delta\tilde{P}=0$. By Liouville's theorem, if $\tilde{P}$ is not constant then it is unbounded, so for any $R>0$ we can choose $x\in\mathbb{R}^{n}$ to make $\tilde{P}(x)=\frac{1}{|B(x,R)|}\int_{B(x,R)}P(y)\,\mathrm{d}y$ arbitrarily large, which implies that the maximal function $P^{**}(\tau)=\sup_{|A|\geq\tau}\frac{1}{|A|}\int_{A}|P(y)|\,\mathrm{d}y$ is infinite for all $\tau\in(0,\infty)$, so $P$ cannot belong to a Lorentz space.
\end{proof}
For the sake of completeness, we give a construction of the Riesz transform. For $\epsilon\in(0,\infty)$ and $n\geq 2$ (a simplified version of the argument exists for $n=1$), we define the \textbf{\textit{truncated Riesz transform}}
\begin{equation}
    \mathcal{R}_{\epsilon}f(x) := \frac{1}{\pi\omega_{n-1}}\int_{|y|>\epsilon}\frac{y}{{|y|}^{n+1}}f(x-y)\,\mathrm{d}y \quad \text{for }f\in L^{2}(\mathbb{R}^{n}),\,x\in\mathbb{R}^{n},
\end{equation}
where $\omega_{n-1}$ is the volume of the unit ball in $\mathbb{R}^{n-1}$. By H\"{o}lder's inequality and continuity of translation in $L^{2}(\mathbb{R}^{n})$, $\mathcal{R}_{\epsilon}$ maps $L^{2}(\mathbb{R}^{n})$ to the space $C_{b,u}^{0}(\mathbb{R}^{n})$ of bounded, uniformly continuous functions. Then $\langle\mathcal{R}_{\epsilon}f,g\rangle=-\langle f,\mathcal{R}_{\epsilon}g\rangle$ for all $f,g\in L^{1}(\mathbb{R}^{n})\cap L^{2}(\mathbb{R}^{n})$ by Fubini's theorem. For $f\in L^{1}(\mathbb{R}^{n})\cap\mathcal{F}L^{1}(\mathbb{R}^{n})$ and $0<\epsilon<R<\infty$, we use Fourier inversion and Fubini's theorem to write
\begin{equation}
    \int_{\epsilon<|y|<R}\frac{y}{{|y|}^{n+1}}f(x-y)\,\mathrm{d}y = \frac{1}{{(2\pi)}^{n}}\int_{\mathbb{R}^{n}}e^{\mathrm{i}x\cdot\xi}\left(\int_{\epsilon<|y|<R}e^{-\mathrm{i}y\cdot\xi}\frac{y}{{|y|}^{n+1}}\,\mathrm{d}y\right)\mathcal{F}f(\xi)\,\mathrm{d}\xi,
\end{equation}
where we use polar coordinates to compute
\begin{equation}
    \int_{\epsilon<|y|<R}e^{-\mathrm{i}y\cdot\xi}\frac{y}{{|y|}^{n+1}}\,\mathrm{d}y = -2\mathrm{i}\left|\mathbb{S}^{n-2}\right|\frac{\xi}{|\xi|}\int_{0}^{\frac{\pi}{2}}\left(\int_{|\xi|\epsilon\cos\theta}^{|\xi|R\cos\theta}\frac{\sin\rho}{\rho}\,\mathrm{d}\rho\right)\cos\theta\sin^{n-2}\theta\,\mathrm{d}\theta.
\end{equation}
Since the improper integral $\int_{0}^{\infty}\frac{\sin\rho}{\rho}\,\mathrm{d}\rho=\frac{\pi}{2}$ exists, we have that $\int_{a}^{b}\frac{\sin\rho}{\rho}\,\mathrm{d}\rho$ defines a bounded continuous function on $\{(a,b)\text{ : }0\leq a\leq b\leq\infty\}$. By dominated convergence, and using the identity $\left|\mathbb{S}^{n-2}\right|\int_{0}^{\frac{\pi}{2}}\cos\theta\sin^{n-2}\theta\,\mathrm{d}\theta=\frac{1}{n-1}\left|\mathbb{S}^{n-2}\right|=\omega_{n-1}$, we deduce that
\begin{equation}
    \mathcal{R}_{\epsilon}f(x) = \frac{1}{{(2\pi)}^{n}}\int_{\mathbb{R}^{n}}e^{\mathrm{i}x\cdot\xi}\frac{-\mathrm{i}\xi}{|\xi|}\lambda_{n}(\epsilon|\xi|)\mathcal{F}f(\xi)\,\mathrm{d}\xi \quad \text{for }f\in L^{1}(\mathbb{R}^{n})\cap\mathcal{F}L^{1}(\mathbb{R}^{n}),\,x\in\mathbb{R}^{n},
\end{equation}
where $\lambda_{n}:[0,\infty)\rightarrow\mathbb{R}$ is a bounded continuous function with $\lambda_{n}(0)=1$.

Now $L^{1}(\mathbb{R}^{n})\cap\mathcal{F}L^{1}(\mathbb{R}^{n})$ is dense in $L^{2}(\mathbb{R}^{n})$, so we deduce that $\mathcal{R}_{\epsilon}f=\mathcal{F}^{-1}\left[\xi\mapsto\frac{-\mathrm{i}\xi}{|\xi|}\lambda_{n}(\epsilon|\xi|)\mathcal{F}f(\xi)\right]$ for all $f\in L^{2}(\mathbb{R}^{n})$. By dominated convergence and continuity of $\mathcal{F}^{-1}:L^{2}(\mathbb{R}^{n})\rightarrow L^{2}(\mathbb{R}^{n})$, it follows that $\mathcal{R}_{\epsilon}f\rightarrow\mathcal{R}_{0}f$ in $L^{2}(\mathbb{R}^{n})$ as $\epsilon\rightarrow0$, where the \textbf{\textit{L\textsuperscript{2} Riesz transform}} is defined by $\mathcal{R}_{0}f:=\mathcal{F}^{-1}\left[\xi\mapsto\frac{-\mathrm{i}\xi}{|\xi|}\mathcal{F}f(\xi)\right]$ for $f\in L^{2}(\mathbb{R}^{n})$. By density of $L^{1}(\mathbb{R}^{n})\cap L^{2}(\mathbb{R}^{n})$ in $L^{2}(\mathbb{R}^{n})$, we have $\langle\mathcal{R}_{0}f,g\rangle=-\langle f,\mathcal{R}_{0}g\rangle$ for all $f,g\in L^{2}(\mathbb{R}^{n})$.

If we can establish the estimate
\begin{equation}\label{calderon-zygmund}
    |\{x\in\mathbb{R}^{n}\text{ : }|\mathcal{R}_{0}f(x)|>\alpha\}| \lesssim_{n} \frac{{\|f\|}_{L^{1}(\mathbb{R}^{n})}}{\alpha} \quad \text{for all }f\in L^{1}(\mathbb{R}^{n})\cap L^{2}(\mathbb{R}^{n}) \text{ and }\alpha\in(0,\infty),
\end{equation}
then the $L^{2}$ Riesz transform extends by density to $\mathcal{R}_{0}:L^{1}(\mathbb{R}^{n})\rightarrow L^{1,\infty}(\mathbb{R}^{n})$. By interpolation (property (ix) from our discussion of Lorentz spaces), it follows that $\mathcal{R}_{0}:L^{p,q}(\mathbb{R}^{n})\rightarrow L^{p,q}(\mathbb{R}^{n})$ for $p\in(1,2)$ and $q\in[1,\infty]$, so by duality the expression $\langle\mathcal{R}_{0}^{*}f,g\rangle:=\langle f,\mathcal{R}_{0}g\rangle$ defines $\mathcal{R}_{0}^{*}:L^{p,\infty}(\mathbb{R}^{n})\rightarrow L^{p,\infty}(\mathbb{R}^{n})$ for $p\in(2,\infty)$. For $p_{0}\in(1,2)$ and $p_{1}\in(2,\infty)$, we have $\mathcal{R}_{0}:L^{p_{0},\infty}(\mathbb{R}^{n})\rightarrow L^{p_{0},\infty}(\mathbb{R}^{n})$ and $\mathcal{R}_{0}^{*}:L^{p_{1},\infty}(\mathbb{R}^{n})\rightarrow L^{p_{1},\infty}(\mathbb{R}^{n})$, with $\mathcal{R}_{0}=-\mathcal{R}_{0}^{*}$ on $L^{p_{0},\infty}(\mathbb{R}^{n})\cap L^{p_{1},\infty}(\mathbb{R}^{n})\subseteq L^{2}(\mathbb{R}^{n})$ (this inclusion follows from property (vii) from our discussion of Lorentz spaces), so the expression $\mathcal{R}(f_{0}+f_{1}):=\mathcal{R}_{0}f_{0}-\mathcal{R}_{0}^{*}f_{1}$ defines a well-defined operator $\mathcal{R}:L^{p_{0},\infty}(\mathbb{R}^{n})+L^{p_{1},\infty}(\mathbb{R}^{n})\rightarrow L^{p_{0},\infty}(\mathbb{R}^{n})+L^{p_{1},\infty}(\mathbb{R}^{n})$ which extends $\mathcal{R}_{0}$ and $-\mathcal{R}_{0}^{*}$. By interpolation, we have $\mathcal{R}:L^{p,q}(\mathbb{R}^{n})\rightarrow L^{p,q}(\mathbb{R}^{n})$ for $p\in(1,\infty)$ and $q\in[1,\infty]$. The operator $\mathcal{R}$ is called the \textbf{\textit{Riesz transform}}, and satisfies $\langle\mathcal{R}f,g\rangle=-\langle f,\mathcal{R}g\rangle$ for $f\in L^{p,q}(\mathbb{R}^{n})$ and $g\in L^{p',q'}(\mathbb{R}^{n})$.

We now prove (\ref{calderon-zygmund}), arguing along the lines of (\cite{stein1970}, sections II.2.4.1, II.2.4.2 and II.3.1). Let $f\in L^{1}(\mathbb{R}^{n})\cap L^{2}(\mathbb{R}^{n})$ and $\alpha\in(0,\infty)$. Since $f\in L^{1}(\mathbb{R}^{n})$, we can use the Calder\'{o}n-Zygmund cube decomposition (\cite{stein1970}, Theorem I.4) to construct a countable collection $\{Q_{k}\}$ of closed hypercubes in $\mathbb{R}^{n}$ (with sides parallel to the coordinate axes) with disjoint interiors such that
\begin{equation}\label{CZ1}
    |f(x)| \leq \alpha \quad \text{for almost every } x\in\mathbb{R}^{n}\setminus\left(\cup_{k}Q_{k}\right),
\end{equation}
\begin{equation}\label{CZ2}
    \alpha < \frac{1}{|Q_{k}|}\int_{Q_{k}}|f(x)|\,\mathrm{d}x \leq 2^{n}\alpha \quad \text{for each hypercube }Q_{k}.
\end{equation}
Writing $\Omega=\cup_{k}Q_{k}$, we use (\ref{CZ2}) to estimate
\begin{equation}\label{CZ3}
    |\Omega| = \sum_{k}|Q_{k}| \leq \sum_{k}\frac{1}{\alpha}\int_{Q_{k}}|f(x)|\,\mathrm{d}x \leq \frac{{\|f\|}_{L^{1}(\mathbb{R}^{n})}}{\alpha}.
\end{equation}
We write $f=g+b$, where
\begin{equation}
    g(x) := \left\{\begin{array}{ll}f(x) & \text{if }x\in\mathbb{R}^{n}\setminus\Omega, \\ \frac{1}{|Q_{k}|}\int_{Q_{k}}f(y)\,\mathrm{d}y & \text{if } x\in\mathrm{int}(Q_{k}). \end{array}\right.
\end{equation}
Here the ``good'' part $g$ is bounded by estimates (\ref{CZ1}) and (\ref{CZ2}), while the ``bad'' part $b$ is supported on the set $\Omega$ of bounded measure (\ref{CZ3}), with $b$ having zero average on each hypercube $Q_{k}$. For $p\in\{1,2\}$ we have
\begin{equation}\label{g-Lp}
\begin{aligned}
    {\|g\|}_{L^{p}(\mathbb{R}^{n})}^{p} &= \int_{\mathbb{R}^{n}\setminus\Omega}{|g(x)|}^{p}\,\mathrm{d}x + \int_{\Omega}{|g(x)|}^{p}\,\mathrm{d}x \\
    &\leq \int_{\mathbb{R}^{n}\setminus\Omega}\alpha^{p-1}|f(x)|\,\mathrm{d}x + |\Omega|2^{np}\alpha^{p} \\
    &\leq (1+2^{np})\alpha^{p-1}{\|f\|}_{L^{1}(\mathbb{R}^{n})}.
\end{aligned}
\end{equation}
Therefore $g$ and $b=f-g$ belong to $L^{2}(\mathbb{R}^{n})$, so we can write $\mathcal{R}_{0}f=\mathcal{R}_{0}g+\mathcal{R}_{0}b$ and we have
\begin{equation}
    |\{x\in\mathbb{R}^{n}\text{ }\mathrm{:}\text{ }|\mathcal{R}_{0}f(x)|>\alpha\}| \leq |\{x\in\mathbb{R}^{n}\text{ }\mathrm{:}\text{ }|\mathcal{R}_{0}g(x)|>\alpha/2\}| + |\{x\in\mathbb{R}^{n}\text{ }\mathrm{:}\text{ }|\mathcal{R}_{0}b(x)|>\alpha/2\}|.
\end{equation}
By Chebychev's inequality and (\ref{g-Lp}) with $p=2$, we have
\begin{equation}
    |\{x\in\mathbb{R}^{n}\text{ }\mathrm{:}\text{ }|\mathcal{R}_{0}g(x)|>\alpha/2\}| \leq \frac{4{\|\mathcal{R}_{0}g\|}_{L^{2}(\mathbb{R}^{n})}^{2}}{\alpha^{2}} = \frac{4{\|g\|}_{L^{2}(\mathbb{R}^{n})}^{2}}{\alpha^{2}} \leq \frac{4(1+2^{2n}){\|f\|}_{L^{1}(\mathbb{R}^{n})}}{\alpha}.
\end{equation}
To estimate $|\{x\in\mathbb{R}^{n}\text{ }\mathrm{:}\text{ }|\mathcal{R}_{0}b(x)|>\alpha/2\}|$, we make the following definition: if the hypercube $Q_{k}$ has centre $y^{(k)}$ and side length $r$, then we let $Q_{k}^{*}$ be the hypercube with centre $y^{(k)}$ and side length $2\sqrt{n}r$. The scaling factor $2\sqrt{n}$ is chosen to ensure that
\begin{equation}\label{scaling}
    |x-y^{(k)}|\geq2|y-y^{(k)}| \quad \text{for all }x\in\mathbb{R}^{n}\setminus Q_{k}^{*}\text{ and }y\in Q_{k}.
\end{equation}
Writing $\Omega^{*}=\cup_{k}Q_{k}^{*}$, we have
\begin{equation}
    |\Omega^{*}| \leq \sum_{k}|Q_{k}^{*}| \lesssim_{n} \sum_{k}|Q_{k}| \leq \frac{{\|f\|}_{L^{1}(\mathbb{R}^{n})}}{\alpha},
\end{equation}
so it remains for us to estimate $|\{x\in\mathbb{R}^{n}\setminus\Omega^{*}\text{ : }|\mathcal{R}_{0}b(x)|>\alpha/2\}|$. Writing $b_{k}(x)=\mathbf{1}_{Q_{k}}(x)b(x)$, we note that the series $b=\sum_{k}b_{k}$ converges absolutely in $L^{2}(\mathbb{R}^{n})$, so the series $\mathcal{R}_{0}b=\sum_{k}\mathcal{R}_{0}b_{k}$ converges absolutely in $L^{2}(\mathbb{R}^{n})$, so $\mathcal{R}_{0}b(x)=\sum_{k}\mathcal{R}_{0}b_{k}(x)$ for almost every $x\in\mathbb{R}^{n}$. By our analysis of truncated Riesz transforms, we have $\mathcal{R}_{0}b_{k}=\lim_{\epsilon\rightarrow0}\mathcal{R}_{\epsilon}b_{k}$ in $L^{2}(\mathbb{R}^{n})$. For all $x\in\mathbb{R}^{n}\setminus Q_{k}^{*}$ we have
\begin{equation}
    \mathcal{R}_{\epsilon}b_{k}(x) := \frac{1}{\pi\omega_{n-1}}\int_{y\in Q_{k},\,|x-y|>\epsilon}\frac{x-y}{{|x-y|}^{n+1}}b(y)\,\mathrm{d}y \overset{\epsilon\rightarrow0}{\rightarrow} \frac{1}{\pi\omega_{n-1}}\int_{Q_{k}}\frac{x-y}{{|x-y|}^{n+1}}b(y)\,\mathrm{d}y.
\end{equation}
Writing $K(z)=\frac{1}{\pi\omega_{n-1}}\frac{z}{{|z|}^{n+1}}$, for almost every $x\in\mathbb{R}^{n}\setminus Q_{k}^{*}$ we deduce that
\begin{equation}
    \mathcal{R}_{0}b_{k}(x)=\int_{Q_{k}}K(x-y)b(y)\,\mathrm{d}y=\int_{Q_{k}}\left(K(x-y)-K(x-y^{(k)})\right)b(y)\,\mathrm{d}y,
\end{equation}
so for almost every $x\in\mathbb{R}^{n}\setminus\Omega^{*}$ we have
\begin{equation}
    \mathcal{R}_{0}b(x) = \sum_{k}\int_{Q_{k}}\left(K(x-y)-K(x-y^{(k)})\right)b(y)\,\mathrm{d}y.
\end{equation}
Therefore
\begin{equation}
\begin{aligned}
    \int_{\mathbb{R}^{n}\setminus\Omega^{*}}|\mathcal{R}_{0}b(x)|\,\mathrm{d}x &\leq \sum_{k}\int_{\mathbb{R}^{n}\setminus Q_{k}^{*}}\left(\int_{Q_{k}}\left|K(x-y)-K(x-y^{(k)})\right||b(y)|\,\mathrm{d}y\right)\,\mathrm{d}x \\
    &\leq \sum_{k}\int_{Q_{k}}\left(\int_{|x-y^{(k)}|\geq2|y-y^{(k)}|}\left|K(x-y)-K(x-y^{(k)})\right|\,\mathrm{d}x\right)|b(y)|\,\mathrm{d}y \\
    &= \sum_{k}\int_{Q_{k}}\left(\int_{|\tilde{x}|\geq2|\tilde{y}|}\left|K(\tilde{x}-\tilde{y})-K(\tilde{x})\right|\,\mathrm{d}\tilde{x}\right)|b(y)|\,\mathrm{d}y \\
    &\leq \sum_{k}\int_{Q_{k}}\left(\int_{|\tilde{x}|\geq2|\tilde{y}|}|\tilde{y}|\sup_{|z|\geq|\tilde{x}|/2}|\nabla K(z)|\,\mathrm{d}\tilde{x}\right)|b(y)|\,\mathrm{d}y \\
    &\lesssim_{n} \sum_{k}\int_{Q_{k}}\left(\int_{|\tilde{x}|\geq2|\tilde{y}|}|\tilde{y}|{|\tilde{x}|}^{-(n+1)}\,\mathrm{d}\tilde{x}\right)|b(y)|\,\mathrm{d}y \\
    &\lesssim_{n} \sum_{k}\int_{Q_{k}}|b(y)|\,\mathrm{d}y \\
    &\lesssim_{n} {\|f\|}_{L^{1}(\mathbb{R}^{n})},
\end{aligned}
\end{equation}
where we use (\ref{scaling}) in the second line, the mean value theorem in the fourth line, and (\ref{g-Lp}) with $p=1$ in the last line. By Chebychev's inequality, we conclude that
\begin{equation}
    |\{x\in\mathbb{R}^{n}\setminus\Omega^{*}\text{ : }|\mathcal{R}_{0}b(x)|>\alpha/2\}| \leq \frac{2}{\alpha}\int_{\mathbb{R}^{n}\setminus\Omega^{*}}|\mathcal{R}_{0}b(x)|\,\mathrm{d}x \lesssim_{n} \frac{{\|f\|}_{L^{1}(\mathbb{R}^{n})}}{\alpha}.
\end{equation}
\subsection{The heat kernel and the mild formulation}
The heat kernel $\Phi$ is defined by
\begin{equation}
    \Phi(t,x) := \frac{1}{{(4\pi t)}^{n/2}}e^{-{|x|}^{2}/4t} = \frac{1}{{(2\pi)}^{n}}\int_{\mathbb{R}^{n}}e^{\mathrm{i}x\cdot\xi-t{|\xi|}^{2}}\,\mathrm{d}\xi \quad \text{for }(t,x)\in(0,\infty)\times\mathbb{R}^{n}.
\end{equation}
One easily verifies that $\Phi$ is a smooth solution to the heat equation $\Phi'=\Delta\Phi$ on $(0,\infty)\times\mathbb{R}^{n}$. The heat kernel also satisfies ${\|\nabla^{\alpha}\Phi(t)\|}_{L^{p',1}(\mathbb{R}^{n})}^{*}\lesssim_{\alpha,n}t^{-\frac{1}{2}\left(|\alpha|+\frac{n}{p}\right)}$ and $\nabla^{\alpha}\Phi(s+t)=\Phi(s)*\nabla^{\alpha}\Phi(t)$ for every $\alpha\in\mathbb{Z}_{\geq0}^{n}$, $p\in(1,\infty]$ and $s,t\in(0,\infty)$ (these well-known results are discussed in \cite{davies2021}).

If $p\in(1,\infty]$ and $f\in L^{p,\infty}(\mathbb{R}^{n})$, then $e^{t\Delta}f(x):=\int_{\mathbb{R}^{n}}\Phi(t,x-y)f(y)\,\mathrm{d}y$ defines a smooth function on $(0,\infty)\times\mathbb{R}^{n}$, satisfying $\nabla^{\alpha}e^{t\Delta}f(x)=\int_{\mathbb{R}^{n}}\nabla^{\alpha}\Phi(t,x-y)f(y)\,\mathrm{d}y$, $\left|\nabla^{\alpha}e^{t\Delta}f(x)\right|\lesssim_{\alpha,n}p't^{-\frac{1}{2}\left(|\alpha|+\frac{n}{p}\right)}{\|f\|}_{L^{p,\infty}(\mathbb{R}^{n})}^{*}$ and ${\|\nabla^{\alpha}e^{t\Delta}f\|}_{L^{p,\infty}(\mathbb{R}^{n})}\lesssim_{\alpha,n} t^{-|\alpha|/2}{\|f\|}_{L^{p,\infty}(\mathbb{R}^{n})}$. The identities $\int_{\mathbb{R}^{n}}\Phi(t,x)\,\mathrm{d}x=1$ and $\Phi(t,x)=t^{-\frac{n}{2}}\Phi(1,\frac{x}{\sqrt{t}})$ will allow us to use approximation of identity results (property (viii) of Lorentz spaces) in the limit $t\rightarrow0$.

If $T\in(0,\infty]$, $p\in(1,\infty]$, and $G:(0,T)\rightarrow L^{p,\infty}(\mathbb{R}^{n})$ is a weakly* measurable function, then $e^{t\Delta}G(s,x):=[e^{t\Delta}G(s)](x)$ satisfies $\nabla^{\alpha}e^{t\Delta}G(s,x)=\langle G(s),\tilde{\Phi}_{\alpha}(t,x)\rangle$, where the expression $\tilde{\Phi}_{\alpha}(t,x;y):=\nabla^{\alpha}\Phi(t,x-y)$ defines a weakly* measurable function $\tilde{\Phi}_{\alpha}:(0,\infty)\times\mathbb{R}^{n}\rightarrow L^{p',1}(\mathbb{R}^{n})$. Approximating $\tilde{\Phi}_{\alpha}$ using Pettis' theorem, it follows that the function $(s,t,x)\mapsto\nabla^{\alpha}e^{t\Delta}G(s,x)$ is measurable. In particular, if $T\in(0,\infty]$, $p\in(1,\infty]$, $\alpha\in\mathbb{Z}_{\geq0}^{n}$, $|\alpha|\leq1$, and $G\in\mathcal{L}_{p,\infty}^{1;*}(T_{-})$, then the expression $\int_{0}^{t}\nabla^{\alpha}e^{(t-s)\Delta}G(s,x)\,\mathrm{d}s$ defines an element of $\mathcal{L}_{p,\infty}^{1}(T_{-})$, satisfying
\begin{equation}
\begin{aligned}
    \int_{0}^{T'}{\left\|\int_{0}^{t}\nabla^{\alpha}e^{(t-s)\Delta}G(s,\cdot)\,\mathrm{d}s\right\|}_{L^{p,\infty}(\mathbb{R}^{n})}\,\mathrm{d}t &\leq \int_{0}^{T'}\left(\int_{0}^{t}{\left\|\nabla^{\alpha}e^{(t-s)\Delta}G(s)\right\|}_{L^{p,\infty}(\mathbb{R}^{n})}\,\mathrm{d}s\right)\,\mathrm{d}t \\
    &\lesssim_{\alpha,n,p} \int_{0}^{T'}\left(\int_{0}^{t}{(t-s)}^{-|\alpha|/2}{\|G(s)\|}_{{\left(L^{p',1}(\mathbb{R}^{n})\right)}^{*}}\,\mathrm{d}s\right)\,\mathrm{d}t \\
    &= \int_{0}^{T'}\left(\int_{s}^{T'}{(t-s)}^{-|\alpha|/2}\,\mathrm{d}t\right){\|G(s)\|}_{{\left(L^{p',1}(\mathbb{R}^{n})\right)}^{*}}\,\mathrm{d}s
\end{aligned}
\end{equation}
for every $T'\in(0,T)$, where we use Minkowski's inequality (property (iv) of Lorentz space) in the first line, the inequality ${\|\nabla^{\alpha}e^{t\Delta}f\|}_{L^{p,\infty}(\mathbb{R}^{n})}\lesssim_{\alpha,n} t^{-|\alpha|/2}{\|f\|}_{L^{p,\infty}(\mathbb{R}^{n})}$ in the second, and Fubini's theorem in the third (having established measurability of ${\|G(s)\|}_{{\left(L^{p',1}(\mathbb{R}^{n})\right)}^{*}}$ in Lemma \ref{measurable-norms}).

We are now in a position to define the \textbf{\textit{mild formulation}} of the linearised Navier-Stokes equations. For $T\in(0,\infty]$, $p_{0},\tilde{p}\in(1,\infty)$, $f\in L^{p_{0},\infty}(\mathbb{R}^{n})$ satisfying $\langle f_{i},\nabla_{i}\phi\rangle$ for all $\phi\in C_{c}^{\infty}(\mathbb{R}^{n})$, and $F\in\mathcal{L}_{\tilde{p},\infty}^{1;*}(T_{-})$, the expression
\begin{equation}
    v_{i}(t,x) := e^{t\Delta}f_{i}(x) - \int_{0}^{t}\nabla_{k}e^{(t-s)\Delta}\mathbb{P}_{ij}F_{jk}(s,x)\,\mathrm{d}s \quad \text{for a.e.\ }(t,x)\in(0,T)\times\mathbb{R}^{n}
\end{equation}
defines an element $v\in\mathcal{L}_{p_{0},\infty}^{\infty}(T)+\mathcal{L}_{\tilde{p},\infty}^{1}(T_{-})$, which we refer to as the \textbf{\textit{mild solution}}.
\begin{theorem}\label{fjrgen}
    Let $T\in(0,\infty]$, $p_{0},p,\tilde{p}\in(1,\infty)$, $f\in L^{p_{0},\infty}(\mathbb{R}^{n})$ with $\langle f_{i},\nabla_{i}\phi\rangle$ for all $\phi\in C_{c}^{\infty}(\mathbb{R}^{n})$, and $F\in\mathcal{L}_{\tilde{p},\infty}^{1;*}(T_{-})$. Then $u\in\mathcal{L}_{p,\infty}^{1;*}(T_{-})$ satisfies the projected formulation of the linearised Navier-Stokes equations if and only if $[u(t)](x)=v(t,x)$ for almost every $(t,x)\in(0,T)\times\mathbb{R}^{n}$, where $v$ is the mild solution.
\end{theorem}
\begin{proof}
    For the purposes of our proof, we choose a fixed function $\rho\in C_{c}^{\infty}(\mathbb{R}^{n})$, which satisfies $\rho(x)=1$ for $|x|<1$ and $\rho(x)=0$ for $|x|>2$. We define $\rho_{R}(x):=\rho(x/R)$ for $R\in(0,\infty)$ and $x\in\mathbb{R}^{n}$. We define $G_{ik}:=\mathbb{P}_{ij}F_{jk}$, $v^{0}(t):=e^{t\Delta}f$, and $v^{1}:=v-v^{0}$.

    We start by verifying that
    \begin{equation}\label{mild-div-free}
        \int_{0}^{T}\langle v_{i}(t),\nabla_{i}\psi(t)\rangle\,\mathrm{d}t=0 \quad \forall\, \psi\in C_{c}^{\infty}((0,T)\times\mathbb{R}^{n}).
    \end{equation}
    By Fubini's theorem, integration by parts, and dominated convergence, we have
    \begin{equation}
        \langle e^{t\Delta}f_{i},\nabla_{i}\phi\rangle =  \langle f_{i},\nabla_{i}e^{t\Delta}\phi\rangle = \lim_{R\rightarrow\infty}\langle f_{i},\nabla_{i}\left[\rho_{R}e^{t\Delta}\phi\right]\rangle = 0
    \end{equation}
    for all $\phi\in C_{c}^{\infty}(\mathbb{R}^{n})$ and $t\in(0,\infty)$, while
    \begin{equation}
        \int_{\mathbb{R}^{n}}\left(\int_{0}^{t}\nabla_{k}e^{(t-s)\Delta}G_{ik}(s,x)\,\mathrm{d}s\right)\nabla_{i}\phi(x)\,\mathrm{d}x = -\lim_{R\rightarrow\infty}\int_{0}^{t}\langle G_{ik}(s),\nabla_{i}\left[\rho_{R}\nabla_{k}e^{(t-s)\Delta}\phi\right]\rangle\,\mathrm{d}s = 0
    \end{equation}
    for all $\phi\in C_{c}^{\infty}(\mathbb{R}^{n})$ for almost every $t\in(0,T)$, where we used the fact that $\mathbb{P}_{ij}\nabla_{i}\theta=0$ for all $\theta\in C_{c}^{\infty}(\mathbb{R}^{n})$. Therefore \eqref{mild-div-free} is verified.

    Next we verify that
    \begin{equation}\label{mild-implies-weak}
        \langle f_{i},\phi_{i}(0)\rangle + \int_{0}^{T}\left(\langle v_{i}(t),\phi_{i}'(t)+\Delta\phi_{i}(t)\rangle + \langle G_{ik}(t),\nabla_{k}\phi_{i}(t)\rangle\right)\,\mathrm{d}t = 0 \quad \forall\,\phi\in C_{c}^{\infty}\left([0,T)\times\mathbb{R}^{n}\right).
    \end{equation}
    By Fubini's theorem, integration by parts, and the identity $\Phi'=\Delta\Phi$, for $\phi\in C_{c}^{\infty}([0,T)\times\mathbb{R}^{n})$ and $(\epsilon,s,x)\in(0,\infty)\times[0,T)\times\mathbb{R}^{n}$ we have
    \begin{equation}
    \begin{aligned}
        \int_{s}^{T}e^{(\epsilon+t-s)\Delta}\phi'(t,x)\,\mathrm{d}t &= \int_{s}^{T}\int_{\mathbb{R}^{n}}\Phi(\epsilon+t-s,x-y)\phi'(t,y)\,\mathrm{d}y\,\mathrm{d}t \\
        &= -\int_{\mathbb{R}^{n}}\Phi(\epsilon,x-y)\phi(s,y)\,\mathrm{d}y - \int_{s}^{T}\int_{\mathbb{R}^{n}}\Phi(\epsilon+t-s,x-y)\Delta\phi(t,y)\,\mathrm{d}y\,\mathrm{d}t \\
        &= -e^{\epsilon\Delta}\phi(s,x) - \int_{s}^{T}e^{(\epsilon+t-s)\Delta}\Delta\phi(t,x)\,\mathrm{d}t.
    \end{aligned}
    \end{equation}
    By Fubini's theorem and the identity $\Phi(\epsilon+t-s)=\Phi(\epsilon)*\Phi(t-s)$, we may write $e^{(\epsilon+t-s)\Delta}=e^{\epsilon\Delta}e^{(t-s)\Delta}$ and take the $e^{\epsilon\Delta}$ outside the integrals to obtain
    \begin{equation}
        e^{\epsilon\Delta}\phi(s,x) + e^{\epsilon\Delta}\int_{s}^{T}e^{(t-s)\Delta}\phi'(t,x)\,\mathrm{d}t + e^{\epsilon\Delta}\int_{s}^{T}e^{(t-s)\Delta}\Delta\phi(t,x)\,\mathrm{d}t = 0
    \end{equation}
    for $\phi\in C_{c}^{\infty}([0,T)\times\mathbb{R}^{n})$ and $(\epsilon,s,x)\in(0,\infty)\times[0,T)\times\mathbb{R}^{n}$. By approximation of identity in $C_{b,u}^{0}(\mathbb{R}^{n})$, we obtain the equality
    \begin{equation}
        \phi(s,x) + \int_{s}^{T}e^{(t-s)\Delta}\phi'(t,x)\,\mathrm{d}t + \int_{s}^{T}e^{(t-s)\Delta}\Delta\phi(t,x)\,\mathrm{d}t = 0
    \end{equation}
    for all $\phi\in C_{c}^{\infty}([0,T)\times\mathbb{R}^{n})$ and $(s,x)\in[0,T)\times\mathbb{R}^{n}$. This last equality allows us to compute
    \begin{equation}
    \begin{aligned}
        \int_{0}^{T}\langle v_{i}^{0}(t),\phi_{i}'(t)+\Delta\phi_{i}(t)\rangle\,\mathrm{d}t &= \int_{0}^{T}\langle e^{t\Delta}f_{i},\phi_{i}'(t)+\Delta\phi_{i}(t)\rangle\,\mathrm{d}t \\
        &= \int_{0}^{T}\langle f_{i},e^{t\Delta}\phi_{i}'(t)+e^{t\Delta}\Delta\phi_{i}(t)\rangle\,\mathrm{d}t \\
        &= -\langle f_{i},\phi_{i}(0)\rangle
    \end{aligned}
    \end{equation}
    and
    \begin{equation}
    \begin{aligned}
        \int_{0}^{T}\langle v_{i}^{1}(t),\phi_{i}'(t)+\Delta\phi_{i}(t)\rangle\,\mathrm{d}t &= -\int_{0}^{T}\left(\int_{0}^{t}\langle\nabla_{k}e^{(t-s)\Delta}G_{ik}(s),\phi_{i}'(t)+\Delta\phi_{i}(t)\rangle\,\mathrm{d}s\right)\,\mathrm{d}t \\
        &= \int_{0}^{T}\left(\int_{s}^{T}\langle G_{ik}(s),\nabla_{k}e^{(t-s)\Delta}\phi_{i}'(t)+\nabla_{k}e^{(t-s)\Delta}\Delta\phi_{i}(t)\rangle\,\mathrm{d}t\right)\,\mathrm{d}s \\
        &= -\int_{0}^{T}\langle G_{ik}(s),\nabla_{k}\phi_{i}(s)\rangle\,\mathrm{d}s
    \end{aligned}
    \end{equation}
    for $\phi\in C_{c}^{\infty}([0,T)\times\mathbb{R}^{n})$. Therefore \eqref{mild-implies-weak} is verified.

    Having established \eqref{mild-div-free} and \eqref{mild-implies-weak}, we have that $u\in\mathcal{L}_{p,\infty}^{1;*}(T_{-})$ satisfies the projected formulation of the linearised Navier-Stokes equations if and only if
    \begin{equation}\label{mild-implies-weak-final}
        \int_{0}^{t}\langle u(t),\phi'(t)+\Delta\phi(t)\rangle\,\mathrm{d}t=\int_{0}^{t}\langle v(t),\phi'(t)+\Delta\phi(t)\rangle\,\mathrm{d}t \quad \forall\,\phi\in C_{c}^{\infty}([0,T)\times\mathbb{R}^{n}).
    \end{equation}
    We claim that \eqref{mild-implies-weak-final} occurs if and only if $[u(t)](x)=v(t,x)$ for almost every $(t,x)\in(0,T)\times\mathbb{R}^{n}$. One direction is obvious, so we prove the other direction. Assume that \eqref{mild-implies-weak-final} holds, and let $\psi\in C_{c}^{\infty}((0,T)\times\mathbb{R}^{n})$ be arbitrary. For $R,\epsilon\in(0,\infty)$, we consider the test function $\phi_{R,\epsilon}\in C_{c}^{\infty}([0,T)\times\mathbb{R}^{n})$ given by
    \begin{equation}
        \phi_{R,\epsilon}(t,x) := -\rho_{R}(x)\int_{t}^{T}e^{(s-t+\epsilon)\Delta}\psi(s,x)\,\mathrm{d}s,
    \end{equation}
    which satisfies
    \begin{equation}
        \phi_{R,\epsilon}'(t,x) = \rho_{m}(x)e^{\epsilon\Delta}\psi(t,x) + \rho_{R}(x)\int_{t}^{T}\Delta e^{(s-t+\epsilon)\Delta}\psi(s,x)\,\mathrm{d}s,
    \end{equation}
    \begin{equation}
    \begin{aligned}
        \Delta\phi_{R,\epsilon}(t,x) &= - \rho_{R}(x)\int_{t}^{T}\Delta e^{(s-t+\epsilon)\Delta}\psi(s,x)\,\mathrm{d}s \\
        &\qquad - 2\nabla\rho_{R}(x)\cdot\int_{t}^{T}\nabla e^{(s-t+\epsilon)\Delta}\psi(s,x)\,\mathrm{d}s \\
        &\qquad - \Delta\rho_{R}(x)\int_{t}^{T}e^{(s-t+\epsilon)\Delta}\psi(s,x)\,\mathrm{d}s,
    \end{aligned}
    \end{equation}
    so for $w=u-v$ we have
    \begin{equation}
    \begin{aligned}
        0 &= \int_{0}^{T}\langle w(t),\phi_{R,\epsilon}'(t)+\Delta\phi_{R,\epsilon}(t)\rangle\,\mathrm{d}t \\
        &= \int_{0}^{T}\langle w(t),\rho_{R}e^{\epsilon\Delta}\psi(t)\rangle\,\mathrm{d}t - 2\int_{0}^{T}\int_{t}^{T}\langle w(t),\nabla\rho_{R}\cdot e^{(s-t+\epsilon)\Delta}\nabla\psi(s)\rangle\,\mathrm{d}s\,\mathrm{d}t \\
        &\qquad - \int_{0}^{T}\int_{t}^{T}\langle w(t),\Delta\rho_{R}e^{(s-t+\epsilon)\Delta}\psi(s)\rangle\,\mathrm{d}s\,\mathrm{d}t.
    \end{aligned}
    \end{equation}
    In the limit $R\rightarrow\infty$ we obtain
    \begin{equation}
        \int_{0}^{T}\langle w(t),e^{\epsilon\Delta}\psi(t)\rangle\,\mathrm{d}t = 0,
    \end{equation}
    so in the limit $\epsilon\rightarrow0$ we obtain
    \begin{equation}
        \int_{0}^{T}\langle w(t),\psi(t)\rangle\,\mathrm{d}t = 0.
    \end{equation}
    Therefore $\int_{0}^{T}\langle u(t),\psi(t)\rangle\,\mathrm{d}t=\int_{0}^{T}\langle v(t),\psi(t)\rangle\,\mathrm{d}t$ for all $\psi\in C_{c}^{\infty}((0,T)\times\mathbb{R}^{n})$, so $[u(t)](x)=v(t,x)$ for almost every $(t,x)\in(0,T)\times\mathbb{R}^{n}$.
\end{proof}
\subsection{Approximation of solenoidal vector fields and the very weak formulation}
For $\phi\in C_{c}^{\infty}(\mathbb{R}^{n};\mathbb{R}^{n})$, the expression $\psi_{i}=\mathbb{P}_{ij}\phi_{j}$ defines a vector field $\psi$ which is smooth and divergence-free, but not necessarily compactly supported. The following lemma allows us to approximate certain smooth, divergence-free vector fields by smooth, divergence-free, compactly supported vector fields.
\begin{lemma}\label{prevlemma}
    Let $p\in(1,\infty)$, $q\in[1,\infty)$ and $k\in\mathbb{Z}_{\geq0}$, and suppose that $\phi\in C^{\infty}(\mathbb{R}^{n};\mathbb{R}^{n})$ satisfies $\nabla\cdot\phi=0$ and ${\|\phi\|}_{W_{p,q}^{k}(\mathbb{R}^{n})}:=\sum_{j=0}^{k}{\|\nabla^{j}\phi\|}_{L^{p,q}(\mathbb{R}^{n})}<\infty$. Then there exist $\phi_{R}\in C_{c}^{\infty}(\mathbb{R}^{n};\mathbb{R}^{n})$ with $\nabla\cdot\phi_{R}=0$ such that ${\|\phi-\phi_{R}\|}_{W_{p,q}^{k}(\mathbb{R}^{n})}\rightarrow0$ as $R\rightarrow\infty$. Moreover, the approximating functions $\phi_{R}$ can be chosen indepently of $p,q,k$.
\end{lemma}
\begin{proof}
    Let $\rho\in C_{c}^{\infty}(B(0,2);\mathbb{R})$ be a fixed function satisfying $\rho(x)=1$ for $|x|<1$, and let $\rho_{R}(x):=\rho(x/R)$. Then $\rho_{R}\phi\in C_{c}^{\infty}(B(0,2R);\mathbb{R}^{n})$ approximates $\phi$ in $W_{p,q}^{k}(\mathbb{R}^{n})$, since for $j\in\{0,\cdots,k\}$ we have
    \begin{equation}
    \begin{aligned}
        {\|\nabla^{j}\left(\left(1-\rho_{R}\right)\phi\right)\|}_{L^{p,q}(\mathbb{R}^{n})} &\lesssim \sum_{i=0}^{j}{\|\nabla^{i}(1-\rho_{R})\nabla^{j-i}\phi\,\mathbf{1}_{|x|>R}\|}_{L^{p,q}(\mathbb{R}^{n})} \\
        &\lesssim \sum_{i=0}^{j}R^{-i}{\|\nabla^{j-i}\phi\,\mathbf{1}_{|x|>R}\|}_{L^{p,q}(\mathbb{R}^{n})} \overset{R\rightarrow\infty}{\rightarrow} 0.
    \end{aligned}
    \end{equation}
    We are not done yet, because $f_{R}:=\nabla\cdot\left(\rho_{R}\phi\right)=\phi\cdot\nabla\rho_{R}$ need not be zero. We will resolve this issue by constructing $v_{R}\in C_{c}^{\infty}(\mathbb{R}^{n};\mathbb{R}^{n})$ satisfying $\nabla\cdot v_{R}=f_{R}$ and ${\|v_{R}\|}_{W_{p,q}^{k}(\mathbb{R}^{n})}\overset{R\rightarrow\infty}{\rightarrow}0$, so that $\phi_{R}=\rho_{R}\phi-v_{R}$ is as required. To this end, let $\omega\in C_{c}^{\infty}(B(0,2);\mathbb{R})$ be a fixed function satisfying $\int_{\mathbb{R}^{n}}\omega(x)\,\mathrm{d}x=1$, and let $\omega_{R}(x):=R^{-n}\omega(x/R)$, so that
    \begin{equation}
    \begin{aligned}
        f_{R}(x) &= \int_{\mathbb{R}^{n}}\left(\omega_{R}(x+y)f_{R}(x)-\omega_{R}(x)f_{R}(x-y)\right)\,\mathrm{d}y \\
        &= -\int_{\mathbb{R}^{n}}\int_{0}^{1}\frac{\partial}{\partial t}\left[\omega_{R}(x+y-ty)f_{R}(x-ty)\right]\,\mathrm{d}t\,\mathrm{d}y \\
        &= \int_{\mathbb{R}^{n}}\int_{0}^{1}y\cdot\frac{\partial}{\partial x}\left[\omega_{R}\left(x+y-ty\right)f_{R}\left(x-ty\right)\right]\,\mathrm{d}t\,\mathrm{d}y.
    \end{aligned}
    \end{equation}
    This last expression is the divergence of a smooth function if we can justify differentiation under the integral. For $\alpha\in\mathbb{Z}_{\geq0}^{n}$ with $|\alpha|=j$ we have
    \begin{equation}
    \begin{aligned}
        &\quad \int_{\mathbb{R}^{n}}\int_{0}^{1}\left|y\frac{\partial^{\alpha}}{{\partial x}^{\alpha}}\left[\omega_{R}\left(x+y-ty\right)f_{R}\left(x-ty\right)\right]\right|\,\mathrm{d}t\,\mathrm{d}y \\
        &\lesssim \sum_{i=0}^{j}\int_{0}^{1}\int_{\mathbb{R}^{n}}\left|y\nabla^{i}\omega_{R}\left(x+y-ty\right)\nabla^{j-i}f_{R}\left(x-ty\right)\right|\,\mathrm{d}y\,\mathrm{d}t \\
        &= \sum_{i=0}^{j}\int_{0}^{1}\int_{\mathbb{R}^{n}}\left|z\nabla^{i}\omega_{R}\left(x+t^{-1}z-z\right)\nabla^{j-i}f_{R}\left(x-z\right)\right|t^{-(n+1)}\,\mathrm{d}z\,\mathrm{d}t \\
        &= \sum_{i=0}^{j}\int_{0}^{\infty}\int_{\mathbb{R}^{n}}\left|z\nabla^{i}\omega_{R}\left(x+rz\right)\nabla^{j-i}f_{R}\left(x-z\right)\right|{(1+r)}^{n-1}\,\mathrm{d}z\,\mathrm{d}r \\
        &= \mathbf{1}_{|x|<2R}\sum_{i=0}^{j}\int_{|z|<4R}\int_{0}^{4R/|z|}\left|z\nabla^{i}\omega_{R}\left(x+rz\right)\nabla^{j-i}f_{R}\left(x-z\right)\right|{(1+r)}^{n-1}\,\mathrm{d}r\,\mathrm{d}z \\
        &\lesssim \mathbf{1}_{|x|<2R}\sum_{i=0}^{j}\int_{|z|<4R}\int_{0}^{4R/|z|}R^{-n-i}\left|z\nabla^{j-i}f_{R}\left(x-z\right)\right|{(1+r)}^{n-1}\,\mathrm{d}r\,\mathrm{d}z \\
        &\lesssim \mathbf{1}_{|x|<2R}\sum_{i=0}^{j}\int_{|z|<4R}R^{-n-i}\left|z\nabla^{j-i}f_{R}\left(x-z\right)\right|\left({\left(1+\frac{4R}{|z|}\right)}^{n}-1\right)\,\mathrm{d}z \\
        &\lesssim \mathbf{1}_{|x|<2R}\sum_{i=0}^{j}\int_{|z|<4R}R^{-i}{|z|}^{1-n}\left|\nabla^{j-i}f_{R}(x-z)\right|\,\mathrm{d}z,
    \end{aligned}
    \end{equation}
    where
    \begin{equation}
        \int_{|z|<4R}R^{-i}{|z|}^{1-n}\,\mathrm{d}z \lesssim R^{1-i}.
    \end{equation}
    We deduce from these estimates that
    \begin{equation}
        v_{R}(x) := \int_{\mathbb{R}^{n}}\int_{0}^{1}y\,\omega_{R}(x+y-ty)f_{R}(x-ty)\,\mathrm{d}t\,\mathrm{d}y
    \end{equation}
    defines $v_{R}\in C_{c}^{\infty}(\mathbb{R}^{n};\mathbb{R}^{n})$ satisfying $\nabla\cdot v_{R}=f_{R}$ and
    \begin{equation}
        \left|\nabla^{j}v_{R}(x)\right| \lesssim \mathbf{1}_{|x|<2R}\sum_{i=0}^{j}\int_{|z|<4R}R^{-i}{|z|}^{1-n}\left|\nabla^{j-i}f_{R}(x-z)\right|\,\mathrm{d}z.
    \end{equation}
    From the convolution inequality ${\|K*f\|}_{L^{p,q}(\mathbb{R}^{n})}\leq{\|K\|}_{L^{1}(\mathbb{R}^{n})}{\|f\|}_{L^{p,q}(\mathbb{R}^{n})}$, it follows for $j\in\{0,\cdots,k\}$ that
    \begin{equation}
    \begin{aligned}
        {\|\nabla^{j}v_{R}\|}_{L^{p,q}(\mathbb{R}^{n})} &\lesssim \sum_{i=0}^{j}R^{1-i}{\|\nabla^{j-i}f_{R}\|}_{L^{p,q}(\mathbb{R}^{n})} \\
        &= \sum_{i=0}^{j}R^{1-i}{\|\nabla^{j-i}\left(\phi\cdot\nabla\rho_{R}\right)\|}_{L^{p,q}(\mathbb{R}^{n})} \\
        &\lesssim \sum_{i=0}^{j}\sum_{l=0}^{j-i}R^{1-i}{\|\nabla^{l+1}\rho_{R}\nabla^{j-i-l}\phi\,\mathbf{1}_{|x|>R}\|}_{L^{p,q}(\mathbb{R}^{n})} \\
        &\lesssim \sum_{i=0}^{j}\sum_{l=0}^{j-i}R^{-i-l}{\|\nabla^{j-i-l}\phi\,\mathbf{1}_{|x|>R}\|}_{L^{p,q}(\mathbb{R}^{n})} \\
        &\lesssim \sum_{m=0}^{j}R^{-m}{\|\nabla^{j-m}\phi\,\mathbf{1}_{|x|>R}\|}_{L^{p,q}(\mathbb{R}^{n})} \overset{R\rightarrow\infty}{\rightarrow} 0. \\
    \end{aligned}
    \end{equation}
\end{proof}
For $T\in(0,\infty]$, $p_{0},p,\tilde{p}\in(1,\infty)$, $f\in L^{p_{0},\infty}(\mathbb{R}^{n})$ satisfying $\langle f_{i},\nabla_{i}\phi\rangle=0$ for all $\phi\in C_{c}^{\infty}(\mathbb{R}^{n})$, and $F\in\mathcal{L}_{\tilde{p},\infty}^{1;*}(T_{-})$, we say that $u\in\mathcal{L}_{p,\infty}^{1;*}(T_{-})$ satisfies the \textbf{\textit{very weak formulation}} of the linearised Navier-Stokes equations if $\int_{0}^{T}\langle u_{j}(t),\nabla_{j}\psi(t)\rangle\,\mathrm{d}t$ for all $\psi\in C_{c}^{\infty}((0,T)\times\mathbb{R}^{n})$, and
\begin{equation}
    \langle f_{i},\theta(0)\phi_{i}\rangle + \int_{0}^{T}\left(\langle u_{i}(t),\theta'(t)\phi_{i}+\theta(t)\Delta\phi_{i}\rangle + \langle F_{jk}(t),\theta(t)\nabla_{k}\phi_{j}\rangle\right)\,\mathrm{d}t = 0
\end{equation}
for all $\theta\in C_{c}^{\infty}([0,T))$ and $\phi\in C_{c,\sigma}^{\infty}(\mathbb{R}^{n})$, where the subscript $\sigma$ means that $\nabla\cdot\phi=0$.
\begin{theorem}
    Let $T\in(0,\infty]$, $p_{0},p,\tilde{p}\in(1,\infty)$, $f\in L^{p_{0},\infty}(\mathbb{R}^{n})$ with $\langle f_{i},\nabla_{i}\phi\rangle=0$ for all $\phi\in C_{c}^{\infty}(\mathbb{R}^{n})$, and $F\in\mathcal{L}_{\tilde{p},\infty}^{1;*}(T_{-})$. Then $u\in\mathcal{L}_{p,\infty}^{1;*}(T_{-})$ satisfies the projected formulation of the linearised Navier-Stokes equations if and only if $u$ satisfies the very weak formulation.
\end{theorem}
\begin{proof}
    Clearly the projected formulation implies the very weak formulation, so we prove the other direction. Assume that $u\in\mathcal{L}_{p,\infty}^{1;*}(T_{-})$ satisfies $\int_{0}^{T}\langle u_{j}(t),\nabla_{j}\psi(t)\rangle\,\mathrm{d}t$ for all $\psi\in C_{c}^{\infty}((0,T)\times\mathbb{R}^{n})$, and
    \begin{equation}
        \langle f_{i},\theta(0)\phi_{i}\rangle + \int_{0}^{T}\left(\langle u_{i}(t),\theta'(t)\phi_{i}+\theta(t)\Delta\phi_{i}\rangle + \langle F_{jk}(t),\theta(t)\nabla_{k}\phi_{j}\rangle\right)\,\mathrm{d}t = 0
    \end{equation}
    for all $\theta\in C_{c}^{\infty}([0,T))$ and $\phi\in C_{c,\sigma}^{\infty}(\mathbb{R}^{n})$. We are then required to show that
    \begin{equation}
        \langle f_{i},\phi_{i}(0)\rangle + \int_{0}^{T}\left(\langle u_{i},\phi_{i}'+\Delta\phi_{i}\rangle + \langle F_{jk},\mathbb{P}_{ij}\nabla_{k}\phi_{i}\rangle\right)\,\mathrm{d}t = 0
    \end{equation}
    for all $\phi\in C_{c}^{\infty}([0,T)\times\mathbb{R}^{n})$. Let $\phi\in C_{c}^{\infty}([0,T)\times\mathbb{R}^{n})$, and consider the smooth, divergence-free vector field $\psi$ given by $\psi_{i}=\mathbb{P}_{ij}\phi_{j}$. Since $u$ and $f$ are weakly divergence-free, we are required to show that
    \begin{equation}\label{proj-very-weak}
        \langle f_{i},\psi_{i}(0)\rangle + \int_{0}^{T}\left(\langle u_{i},\psi_{i}'+\Delta\psi_{i}\rangle + \langle F_{jk},\nabla_{k}\psi_{j}\rangle\right)\,\mathrm{d}t = 0.
    \end{equation}
    Define the norm ${\|g\|}_{X}:={\|g\|}_{L^{p_{0}',1}(\mathbb{R}^{n})}\vee{\|g\|}_{W_{p',1}^{2}(\mathbb{R}^{n})}\vee{\|\nabla g\|}_{L^{\tilde{p}',1}(\mathbb{R}^{n})}$ (where the norm ${\|\cdot\|}_{W_{p',1}^{2}(\mathbb{R}^{n})}$ was defined in the previous lemma), and let $\epsilon\in(0,\infty)$ be arbitrary. The map $\psi'':[0,T)\rightarrow X$ is continuous and compactly supported, so there exists a compactly supported simple function $\alpha:[0,T)\rightarrow X$, taking values in the range of $\psi''$, such that ${\|\psi''(t)-\alpha(t)\|}_{X}<\epsilon$ for all $t\in[0,T)$. By Lemma \ref{prevlemma}, there exists a compactly supported simple function $\beta:[0,T)\rightarrow C_{c,\sigma}^{\infty}(\mathbb{R}^{n})$ such that ${\|\alpha(t)-\beta(t)\|}_{X}<\epsilon$ for all $t\in[0,T)$. We then define $\gamma(t,x) := -\int_{t}^{T}\beta(s,x)\,\mathrm{d}s$ and $\delta(t,x) := -\int_{t}^{T}\gamma(s,x)\,\mathrm{d}s$, which satisfy
    \begin{equation}\label{gamma-estimate}
        {\|\psi'(t)-\gamma(t)\|}_{X} \leq \int_{t}^{T'}{\|\psi''(s)-\beta(s)\|}_{X}\,\mathrm{d}s < 2\epsilon(T'-t),
    \end{equation}
    \begin{equation}\label{delta-estimate}
        {\|\psi(t)-\delta(t)\|}_{X} \leq \int_{t}^{T'}{\|\psi'(s)-\gamma(s)\|}_{X}\,\mathrm{d}s < \epsilon{(T'-t)}^{2},
    \end{equation}
    where $T'\in(0,T)$ is such that $\psi,\alpha,\beta$ are compactly supported within $[0,T')$. Now $\delta$ is a finite sum of functions of the form $\theta(t)\chi(x)$, where $\theta\in C_{c}^{1}([0,T))$ and $\chi\in C_{c,\sigma}^{\infty}(\mathbb{R}^{n})$. Approximating each of the functions $\theta$ by a smooth function, and using the assumption that $u$ satisfies the very weak formulation, we obtain
    \begin{equation}
        \langle f_{i},\delta_{i}(0)\rangle + \int_{0}^{T}\left(\langle u_{i},\gamma_{i}+\Delta\delta_{i}\rangle + \langle F_{jk},\nabla_{k}\delta_{j}\rangle\right)\,\mathrm{d}t = 0,
    \end{equation}
    where we note that $\delta'=\gamma$. Using estimates \eqref{gamma-estimate} and \eqref{delta-estimate}, we deduce that the two sides of \eqref{proj-very-weak} differ by $O(\epsilon)$, which shrinks to zero as $\epsilon\rightarrow0$.
\end{proof}
\section{Energy estimates and weak-strong uniqueness}
\subsection{The energy class}
The \textbf{\textit{Sobolev space}} $H^{1}(\mathbb{R}^{n})$ consists of those $u\in L^{2}(\mathbb{R}^{n})$ for which there exists $v\in L^{2}(\mathbb{R}^{n})$ satisfying $\langle u,\nabla\phi\rangle=-\langle v,\phi\rangle$ for all $\phi\in C_{c}^{\infty}(\mathbb{R}^{n})$. For $u\in H^{1}(\mathbb{R}^{n})$, we write $\nabla u$ to denote the unique $v\in L^{2}(\mathbb{R}^{n})$ appearing in the definition. By considering $(u*\eta_{\epsilon})\rho_{R}$, where $\rho_{R}$ is a smooth cutoff function and $\eta_{\epsilon}$ is a smooth approximate identity, we see that $C_{c}^{\infty}(\mathbb{R}^{n})$ is dense in $H^{1}(\mathbb{R}^{n})$, where ${\|u\|}_{H^{1}(\mathbb{R}^{n})}:={\|(u,\nabla u)\|}_{L^{2}(\mathbb{R}^{n})}$. From (\cite{brezis2010}, p.\ 280) we have
\begin{equation}\label{pre-sobolev-inequality}
    {\|u\|}_{L^{\frac{\alpha n}{n-1}}(\mathbb{R}^{n})}^{\alpha} \leq \alpha{\|u\|}_{L^{2(\alpha-1)}(\mathbb{R}^{n})}^{\alpha-1}{\|\nabla u\|}_{L^{2}(\mathbb{R}^{n})} \quad \text{for }n\geq 2\text{ and }\alpha\in\left[\frac{3}{2},\infty\right),
\end{equation}
from which we deduce the \textbf{\textit{Sobolev inequality}}
\begin{equation}
    {\|u\|}_{L^{\frac{2p}{p-2},2}(\mathbb{R}^{n})} \lesssim_{n,p} {\|u\|}_{L^{2}(\mathbb{R}^{n})}^{1-\frac{n}{p}}{\|\nabla u\|}_{L^{2}(\mathbb{R}^{n})}^{\frac{n}{p}} \quad \text{for }n\geq2\text{ and }p\in(n,\infty].
\end{equation}
(In the case $n>2$, take (\ref{pre-sobolev-inequality}) with $\alpha=\frac{2(n-1)}{n-2}$, and apply property (vii) from our discussion of Lorentz spaces; in the case $n=2$, use (\ref{pre-sobolev-inequality}) to prove ${\|u\|}_{L^{2k}(\mathbb{R}^{2})}^{k}\leq k!{\|u\|}_{L^{2}(\mathbb{R}^{2})}{\|\nabla u\|}_{L^{2}(\mathbb{R}^{2})}^{k-1}$ by induction on $k\in\mathbb{N}$, then apply property (vii) of Lorentz spaces).

For $n\geq2$ and $T\in(0,\infty]$, the \textbf{\textit{energy class}} $\mathcal{H}_{T}$ consists of those (equivalence classes of) weakly* measurable functions $u:(0,T)\rightarrow H^{1}(\mathbb{R}^{n})\subseteq L_{\mathrm{loc}}^{1}(\mathbb{R}^{n})$ satisfying $u\in\mathcal{L}_{2}^{\infty;*}(T_{-})$ and $\nabla u\in\mathcal{L}_{2}^{2;*}(T_{-})$, where we observe that weak* measurability of $u$ implies weak* measurability of $\nabla u$. By virtue of the identification $\mathcal{L}_{2}^{2;*}(T')\cong{\left(L^{2}((0,T')\times\mathbb{R}^{n})\right)}^{*}\cong L^{2}((0,T')\times\mathbb{R}^{n})$, we may identify $u$ and $\nabla u$ with measurable functions $u\in\mathcal{L}_{2}^{\infty}(T_{-})$ and $\nabla u\in\mathcal{L}_{2}^{2}(T_{-})$. The Sobolev inequality implies that $\mathcal{H}_{T}\subseteq\cap_{p\in(n,\infty]}\mathcal{L}_{\frac{2p}{p-2},2}^{2p/n}(T_{-})$, so we can apply our results from the previous section on the various formulations of the Navier-Stokes equations.

The following result concerns continuity at the initial time for soutions in the energy class.
\begin{theorem}
    If $T\in(0,\infty]$, $f\in L^{2}(\mathbb{R}^{n})$ satisfies $\langle f_{i},\nabla_{i}\phi\rangle$ for all $\phi\in C_{c}^{\infty}(\mathbb{R}^{n})$, and $u\in\mathcal{H}_{T}$ satisfies the mild formulation of the Navier-Stokes equations, then $u$ satisfies the continuity condition
    \begin{equation}\label{continuity-condition}
        \left\{\begin{array}{l}\text{there exists a subset }\Omega\subseteq(0,T)\text{ of total measure such that} \\ \langle u(t),\phi\rangle\rightarrow\langle f,\phi\rangle\text{ for all }\phi\in L^{2}(\mathbb{R}^{n})\text{ as }t\rightarrow0\text{ along }\Omega. \end{array}\right.
    \end{equation}
\end{theorem}
\begin{proof}
    The mild solution $u$ satisfies
    \begin{equation}\label{energy-mild-solution}
        u_{i}(t,x) = e^{t\Delta}f_{i}(x) - \int_{0}^{t}e^{(t-s)\Delta}\mathbb{P}_{ij}\left[u_{k}(s)\nabla_{k}u_{j}(s)\right](x)\,\mathrm{d}s \quad \text{for a.e.\ }(t,x)\in(0,T)\times\mathbb{R}^{n},
    \end{equation}
    where the transfer of the derivative onto $u$ is justified using integration by parts, approximating $u(s)$ by smooth compactly supported functions, and using the fact that $\langle u_{i}(s),\nabla_{i}\phi\rangle=0$ for all $\phi\in C_{c}^{\infty}(\mathbb{R}^{n})$ for almost every $s\in(0,T)$. Now ${\|e^{t\Delta}f-f\|}_{L^{2}(\mathbb{R}^{n})}\overset{t\rightarrow0}{\rightarrow}0$ by approximation of identity, while by Minkowski's inequality (property (iv) of Lorentz spaces), properties of the heat kernel and Riesz transform, and the Sobolev inequality, for $p\in(n,\infty)$ we have
    \begin{equation}
        {\left\|\int_{0}^{t}e^{(t-s)\Delta}\mathbb{P}_{ij}\left[u_{k}(s)\nabla_{k}u_{j}(s)\right](\cdot)\,\mathrm{d}s\right\|}_{L^{\frac{p}{p-1},1}(\mathbb{R}^{n})} \lesssim_{n,p} \int_{0}^{t}{\|u(s)\|}_{L^{\frac{2p}{p-2},2}(\mathbb{R}^{n})}{\|\nabla u(s)\|}_{L^{2}(\mathbb{R}^{n})}\,\mathrm{d}s \overset{t\rightarrow0}{\rightarrow}0.
    \end{equation}
    Therefore $\lim_{t\in\Omega,t\rightarrow0}\langle u(t),\phi\rangle=\langle f,\phi\rangle$ for all $\phi\in C_{c}^{\infty}(\mathbb{R}^{n})$, where $t\in\Omega$ iff (\ref{energy-mild-solution}) holds for almost every $x\in\mathbb{R}^{n}$. We conclude by density of $C_{c}^{\infty}(\mathbb{R}^{n})$ in $L^{2}(\mathbb{R}^{n})$, and by the fact that ${\|u(t)\|}_{L^{2}(\mathbb{R}^{n})}$ is essentially bounded near $t=0$.
\end{proof}
\subsection{Energy estimates}
For $p\in(n,\infty]$ and $T\in(0,\infty]$ we define $\mathcal{V}_{T}^{p}:=\mathcal{H}_{T}\cap\mathcal{L}_{p,\infty}^{2p/(p-n)}(T_{-})$. In the particular case $n=2$, we have $\mathcal{H}_{T}\subseteq\cap_{p\in(2,\infty)}\mathcal{V}_{T}^{p}$ by the Sobolev inequality.
\begin{theorem}\label{energy-theorem}
    Let $T\in(0,\infty]$, and assume that $f\in L^{2}(\mathbb{R}^{n})$ satisfies $\langle f_{i},\nabla_{i}\phi\rangle=0$ for all $\phi\in C_{c}^{\infty}(\mathbb{R}^{n})$. Assume that $p\in(n,\infty]$, and that $u\in\mathcal{H}_{T}\cap\mathcal{V}_{T}^{p}$ and $v\in\mathcal{H}_{T}$ satisfy the projected formulation of the Navier-Stokes equations with intial data $f$, with $u$ and $v$ both satisfying the continuity condition \eqref{continuity-condition}. Then for almost every $t\in(0,T)$ we have
    \begin{equation}\label{energy}
    \begin{aligned}
        \langle u_{i}(t),v_{i}(t)\rangle &= {\|f\|}_{L^{2}(\mathbb{R}^{n})}^{2} - 2\int_{0}^{t}\langle\nabla_{j}u_{i}(s),\nabla_{j}v_{i}(s)\rangle\,\mathrm{d}s \\
        &\qquad -\int_{0}^{t}\langle{[(u\cdot\nabla)u]}_{i}(s),v_{i}(s)\rangle\,\mathrm{d}s -\int_{0}^{t}\langle u_{i}(s),{[(v\cdot\nabla)v]}_{i}(s)\rangle\,\mathrm{d}s.
    \end{aligned}
    \end{equation}
    In particular, $u$ satisfies the \textbf{\textit{energy equality}}
    \begin{equation}\label{energy-equality}
        {\|f\|}_{L^{2}(\mathbb{R}^{n})}^{2} = {\|u(t)\|}_{L^{2}(\mathbb{R}^{n})}^{2} + 2\int_{0}^{t}{\|\nabla u(s)\|}_{L^{2}(\mathbb{R}^{n})}^{2}\,\mathrm{d}s \quad \text{for a.e.\ }t\in(0,T).
    \end{equation}
\end{theorem}
\begin{proof}
    By the Sobolev inequality, and by the assmuptions $u\in\mathcal{V}_{T}^{p}$ and $v\in\mathcal{H}_{T}$, we have
    \begin{equation}\label{regularity}
        \begin{array}{ccc} u\in\mathcal{L}_{2}^{\infty}(T_{-})\cap\mathcal{L}_{p,\infty}^{2p/(p-n)}(T_{-}), & \nabla u\in \mathcal{L}_{2}^{2}(T_{-}), & (u\cdot\nabla)u\in \mathcal{L}_{\frac{2p}{p+2},2}^{2p/(2p-n)}(T_{-}), \\ v\in \mathcal{L}_{2}^{\infty}(T_{-})\cap\mathcal{L}_{\frac{2p}{p-2},2}^{2p/n}(T_{-}), & \nabla v\in \mathcal{L}_{2}^{2}(T_{-}), & (v\cdot\nabla)v\in \mathcal{L}_{\frac{p}{p-1},1}^{2p/(p+n)}(T_{-}), \end{array}
    \end{equation}
    so (\ref{energy}) makes sense. To prove (\ref{energy}), let $\eta\in C_{c}^{\infty}((-1,1))$ with $\int_{-1}^{1}\eta(t)\,\mathrm{d}t=1$, let $\eta_{\epsilon}(t)=\epsilon^{-1}\eta(t/\epsilon)$, and define the regularised functions
    \begin{equation}
        u^{\epsilon}(t,x) := \int_{0}^{T}\eta_{\epsilon}(t-s)u(s,x)\,\mathrm{d}s, \quad v^{\epsilon}(t,x) := \int_{0}^{T}\eta_{\epsilon}(t-s)v(s,x)\,\mathrm{d}s
    \end{equation}
    for all $t\in(-\infty,T-\epsilon)$ and almost every $x\in\mathbb{R}^{n}$  (by Minkowski's inequality, the functions $u^{\epsilon}$ and $v^{\epsilon}$ are well-defined, with $u^{\epsilon}(t)\in H^{1}(\mathbb{R}^{n})\cap L^{p,\infty}(\mathbb{R}^{n})$ and $v^{\epsilon}(t)\in H^{1}(\mathbb{R}^{n})$). Differentiating under the integral, for all $t\in(-\infty,T-\epsilon)$ we have
    \begin{equation}\label{mixed-derivative}
        \frac{\partial}{\partial t}\langle u_{i}^{\epsilon}(t),v_{i}^{\epsilon}(t)\rangle = \int_{0}^{T}\eta_{\epsilon}'(t-s)\langle u_{i}(s),v_{i}^{\epsilon}(t)\rangle\,\mathrm{d}s + \int_{0}^{T}\eta_{\epsilon}'(t-s)\langle u_{i}^{\epsilon}(t),v_{i}(s)\rangle\,\mathrm{d}s.
    \end{equation}
    Since $u$ satisfies the projected formulation of the Navier-Stokes equations, for $\theta\in C_{c}^{\infty}([0,T))$ and $\psi\in C_{c}^{\infty}(\mathbb{R}^{n})$ we have
    \begin{equation}
        \theta(0)\langle f_{i},\psi_{i}\rangle + \int_{0}^{T}\left(\langle u_{i}(s),\theta'(s)\psi_{i}\rangle-\langle\nabla_{j}u_{i}(s),\theta(s)\nabla_{j}\psi_{i}\rangle - \langle{[(u\cdot\nabla)u]}_{i}(s),\theta(s)\mathbb{P}_{ij}\psi_{i}\rangle\right)\,\mathrm{d}s=0.
    \end{equation}
    For $t\in(-\infty,T-\epsilon)$ and $\psi\in C_{c}^{\infty}(\mathbb{R}^{n})$, if we take $\theta(s)=\eta_{\epsilon}(t-s)$ in this last equality then we obtain
    \begin{equation}\label{strong-derivative}
        \int_{0}^{T}\eta_{\epsilon}'(t-s)\langle u_{i}(s),\psi_{i}\rangle\,\mathrm{d}s=\eta_{\epsilon}(t)\langle f_{i},\psi_{i}\rangle - \langle\nabla_{j}u_{i}^{\epsilon}(t),\nabla_{j}\psi_{i}\rangle - \langle{[(u\cdot\nabla)u]}_{i}^{\epsilon}(t),\mathbb{P}_{ij}\psi_{i}\rangle.
    \end{equation}
    Analogously, for $t\in(-\infty,T-\epsilon)$ and $\psi\in C_{c}^{\infty}(\mathbb{R}^{n})$ we have
    \begin{equation}\label{weak-derivative}
        \int_{0}^{T}\eta_{\epsilon}'(t-s)\langle v_{i}(s),\psi_{i}\rangle\,\mathrm{d}s=\eta_{\epsilon}(t)\langle f_{i},\psi_{i}\rangle - \langle\nabla_{j}v_{i}^{\epsilon}(t),\nabla_{j}\psi_{i}\rangle - \langle{[(v\cdot\nabla)v]}_{i}^{\epsilon}(t),\mathbb{P}_{ij}\psi_{i}\rangle.
    \end{equation}
    Noting that ${[(u\cdot\nabla)u]}^{\epsilon}(t)\in L^{\frac{2p}{p+2},2}(\mathbb{R}^{n})$ and ${[(v\cdot\nabla)v]}^{\epsilon}(t)\in L^{\frac{p}{p-1},1}(\mathbb{R}^{n})$ (with a minor modification in the case $p=\infty$ to deal with the Leray projection), we see that (\ref{strong-derivative}) extends to all $\psi\in H^{1}(\mathbb{R}^{n})\subseteq L^{\frac{2p}{p-2},2}(\mathbb{R}^{n})$, while (\ref{weak-derivative}) extends to all $\psi\in H^{1}(\mathbb{R}^{n})\cap L^{p,\infty}(\mathbb{R}^{n})$ (replace $\psi$ by $(\phi_{\delta}*\psi)\rho_{R}$, where $\phi_{\delta}$ is an approximate identity and $\rho_{R}$ is a smooth cutoff function, and take the limit $R\rightarrow\infty$ then $\delta\rightarrow0$). By (\ref{mixed-derivative}), we deduce for $t\in(-\infty,T-\epsilon)$ that
    \begin{equation}\label{approx-energy}
    \begin{aligned}
        \langle u_{i}^{\epsilon}(t),v_{i}^{\epsilon}(t)\rangle &= \int_{-\epsilon}^{t}\eta_{\epsilon}(s)\langle f_{i},u_{i}^{\epsilon}(s)+v_{i}^{\epsilon}(s)\rangle\,\mathrm{d}s - 2\int_{-\epsilon}^{t}\langle\nabla_{j}u_{i}^{\epsilon}(s),\nabla_{j}v_{i}^{\epsilon}(s)\rangle\,\mathrm{d}s \\
        &\qquad - \int_{-\epsilon}^{t}\langle{[(u\cdot\nabla)u]}_{i}^{\epsilon}(s),v_{i}^{\epsilon}(s)\rangle\,\mathrm{d}s - \int_{-\epsilon}^{t}\langle u_{i}^{\epsilon}(s),{[(v\cdot\nabla)v]}_{i}^{\epsilon}(s)\rangle\,\mathrm{d}s.
    \end{aligned}
    \end{equation}
    By \eqref{continuity-condition} and dominated convergence, for $t\in(\epsilon,T-\epsilon)$ we have
    \begin{equation}
        \int_{-\epsilon}^{t}\eta_{\epsilon}(s)\langle f_{i},u_{i}^{\epsilon}(s)\rangle\,\mathrm{d}s = \int_{-1}^{1}\int_{-1}^{\sigma}\eta(\sigma)\eta(\rho)\langle f_{i},u_{i}(\epsilon(\sigma-\rho))\rangle\,\mathrm{d}\rho\,\mathrm{d}\sigma\overset{\epsilon\rightarrow0}{\rightarrow}\frac{1}{2}{\|f\|}_{L^{2}(\mathbb{R}^{n})}^{2},
    \end{equation}
    so the first integral on the right hand side of (\ref{approx-energy}) converges to ${\|f\|}_{L^{2}(\mathbb{R}^{n})}^{2}$. By Lemma \ref{approximation-lemma} (defining $u(t)=v(t)=0$ for $t<0$, and noting the regularity described in (\ref{regularity})), the last three integrals on the right hand side of (\ref{approx-energy}) converge to the last three integrals on the right hand side of (\ref{energy}), while $\langle u^{\epsilon_{m}}(t),v^{\epsilon_{m}}(t)\rangle\rightarrow\langle u(t),v(t)\rangle$ for almost every $t\in(0,T)$ along some subsequence $\epsilon_{m}\rightarrow0$. We conclude that (\ref{energy}) holds for almost every $t\in(0,T)$.
\end{proof}
\subsection{Weak-strong uniqueness}
The estimates of Theorem \ref{energy-theorem} allow us to prove the following weak-strong uniqueness result.
\begin{theorem}
    Let $T\in(0,\infty]$, and assume that $f\in L^{2}(\mathbb{R}^{n})$ satisfies $\langle f_{i},\nabla_{i}\phi\rangle=0$ for all $\phi\in C_{c}^{\infty}(\mathbb{R}^{n})$. Assume that $p\in(n,\infty]$, and that $u\in\mathcal{H}_{T}\cap\mathcal{V}_{T}^{p}$ and $v\in\mathcal{H}_{T}$ satisfy the projected formulation of the Navier-Stokes equations with intial data $f$, with $u$ and $v$ both satisfying the continuity condition \eqref{continuity-condition}, and $v$ satisfying the \textbf{\textit{energy inequality}}
    \begin{equation}\label{energy-inequality}
        {\|f\|}_{L^{2}(\mathbb{R}^{n})}^{2}\geq{\|v(t)\|}_{L^{2}(\mathbb{R}^{n})}^{2} + 2\int_{0}^{t}{\|\nabla v(s)\|}_{L^{2}(\mathbb{R}^{n})}^{2}\,\mathrm{d}s \quad \text{for a.e.\ }t\in(0,T).
    \end{equation}
    Then $u(t,x)=v(t,x)$ for almost every $(t,x)\in(0,T)\times\mathbb{R}^{n}$.
\end{theorem}
\begin{proof}
    Let $w=u-v$. By (\ref{energy}), (\ref{energy-equality}) and (\ref{energy-inequality}), for almost every $t\in(0,T)$ we have
    \begin{equation}
    \begin{aligned}
        {\|w(t)\|}_{L^{2}(\mathbb{R}^{n})}^{2} + 2\int_{0}^{t}{\|\nabla w(s)\|}_{L^{2}(\mathbb{R}^{n})}^{2}\,\mathrm{d}s &\leq 2\int_{0}^{t}\langle{[(u\cdot\nabla)u]}_{i}(s),v_{i}(s)\rangle\,\mathrm{d}s + 2\int_{0}^{t}\langle u_{i}(s),{[(v\cdot\nabla)v]}_{i}(s)\rangle\,\mathrm{d}s \\
        &= -2\int_{0}^{t}\langle u_{i}(s),{[(w\cdot\nabla)w]}_{i}(s)\rangle\,\mathrm{d}s \\
        &\lesssim_{n,p} \int_{0}^{t}{\|u(s)\|}_{L^{p,\infty}(\mathbb{R}^{n})}{\|w(s)\|}_{L^{\frac{2p}{p-2},2}(\mathbb{R}^{n})}{\|\nabla w(s)\|}_{L^{2}(\mathbb{R}^{n})}\,\mathrm{d}s \\
        &\lesssim_{n,p} \int_{0}^{t}{\|u(s)\|}_{L^{p,\infty}(\mathbb{R}^{n})}{\|w(s)\|}_{L^{2}(\mathbb{R}^{n})}^{\frac{p-n}{p}}{\|\nabla w(s)\|}_{L^{2}(\mathbb{R}^{n})}^{\frac{p+n}{p}}\,\mathrm{d}s \\
        &\lesssim_{n,p} \frac{1}{\epsilon}\int_{0}^{t}{\|u(s)\|}_{L^{p,\infty}(\mathbb{R}^{n})}^{\frac{2p}{p-n}}{\|w(s)\|}_{L^{2}(\mathbb{R}^{n})}^{2}\,\mathrm{d}s + \epsilon\int_{0}^{t}{\|\nabla w(s)\|}_{L^{2}(\mathbb{R}^{n})}^{2}\,\mathrm{d}s.
    \end{aligned}
    \end{equation}
    Taking $\epsilon>0$ sufficiently small, we deduce that
    \begin{equation}
        {\|w(t)\|}_{L^{2}(\mathbb{R}^{n})}^{2} \lesssim_{n,p} \int_{0}^{t}{\|u(s)\|}_{L^{p,\infty}(\mathbb{R}^{n})}^{\frac{2p}{p-n}}{\|w(s)\|}_{L^{2}(\mathbb{R}^{n})}^{2}\,\mathrm{d}s \quad \text{for a.e.\ }t\in(0,T),
    \end{equation}
    so by Gr\"{o}nwall's inequality we have ${\|w(t)\|}_{L^{2}(\mathbb{R}^{n})}=0$ for almost every $t\in(0,T)$.
\end{proof}

\end{document}